\documentclass[12pt,a4paper]{article}
\usepackage[utf8]{inputenc}
\usepackage[english]{babel}
\usepackage{amsmath}
\usepackage{amsfonts}
\usepackage{graphicx}
\usepackage{fouridx}
\usepackage{amsthm}
\usepackage{amssymb, amscd}
\usepackage{enumerate}
\usepackage{mathtools}
\usepackage[french]{varioref}
\usepackage{multirow} 
\usepackage{hyperref}
\usepackage{bm}
\author{  Saadaoui Nejib \thanks{ Université de Gabès , Institut Supérieur d’Informatique de Médenine,  Route Djerba Km 3, Boite Postale N 283, 4100 Medenine, Tunis,\qquad \textbf{nejib.saadaoui@fsg.rnu.tn}}}
\title{Cohomology of Heisenberg Hom-Lie algebras}
\newtheorem{theorem}{Theorem}[section]
\newtheorem{cor}[theorem]{Corollary}
\newtheorem{thm}[theorem]{Theorem}
\newtheorem{lem}[theorem]{Lemma}
\newtheorem{prop}[theorem]{Proposition}
\newtheorem{example}[theorem]{Example}
\newtheorem{defn}[theorem]{Definition}
\newtheorem{remq}[theorem]{Remark}
\numberwithin{equation}{section}
\usepackage{multicol}
\input{epsf.sty}
\newtheorem{Case}{\textbf{Case}}

\usepackage{varioref}

\newcommand{\N}{\mathbb{N}}
\newcommand{\C}{\mathbb{C}}
\newcommand{\Z}{\mathbb{Z}}
\newcommand{\R}{\mathbb{R}}

\newcommand{\crm}[1]{\romannumeral #1}
\newcommand{\intervalle}[4]{\mathopen{#1}#2
	\mathclose{}\mathpunct{};#3
	\mathclose{#4}}
\newcommand{\intervalleff}[2]{\intervalle{[}{#1}{#2}{]}}

\newcommand{\intervallefo}[2]{\intervalle{[}{#1}{#2}{[}}

\begin{document} 
	\maketitle
	\begin{abstract}
	In this paper, we define the Heisenberg Hom-Lie algebra. We determine the minimal dimension of faithful representation for Heisenberg Hom-Lie algebra.We study the  adjoint representation, the trivial representation and the faithful representation of Heisenberg Hom-Lie algebra. \\

		\begin{small}\textbf{Keywords} : Heisenberg Hom-Lie algebra, Symplectic, Oscillator, Faithful representation, extension, derivation, cohomology.\end{small}
	\end{abstract}
\section*{Introduction}
Heisenberg Algebras are used in many fields in mathematics and physics.
In classical mechanics the state of a particle at a given time $t$ is determined
by its position vector $q \in \R^3$ and its momentum vector $p \in \R^3$
. Heisenberg’s
crucial idea that lead to quantum mechanics was to take the components of
these vectors to be operators on a Hilbert space $ \mathcal{H} $, satisfying the commutation
relations
\[ [Q_i,Q_j]=0,\, [P_i,P_j]=0, \ [P_i,Q_j]=-i\,\,\hslash\,\delta_{i,j}\]
for $ i,\ j=1,\, 2,\, 3. $\\
   An algebra generated by $2n+1$ elements $\{P_1,\cdots,P_n, Q_1, \cdots, Q_n,\,\hslash\}$ satisfying
the relations above will be called a Heisenberg algebra and denoted $ \hslash(n) $. For more details, see
 \cite{quantum,Mechanics,polynome,introduction}.

The paper is organized as follows. Section $1$ reviews definitions and properties of Hom-Lie algebras.  In section $ 2 $, we define $ \lambda $-
symplectic vector spaces.
In section $3$, we define  Hom-Lie Heisenberg Algebras and we give a classification of Hom-Lie Heisenberg Algebras. In section $ 4, $ we recall the definition of $ \alpha^r $-derivation of Hom-Lie algebras, we compute the set of $\alpha^r$-derivation of Hom-Lie Heisenberg Algebras, we define the Hom-Lie algebra $ \mathcal{H}(D) $ and the meta Heisenberg Hom algebra.
Sections  $ 5, $ $ 6 $ and $ 7 $ deal respectively the faithfull, trivial and adjoint representation.
\section{Preliminaries}
\subsection{Hom-Lie algebras}
\begin{defn}\cite{HLS,Ayedi,AbdenacerM}
A Hom-Lie algebra  is a triple $(\mathcal{G},\ [\cdot,\cdot],\ \alpha)$\ consisting of a $ \mathbb{K} $ vector space $\mathcal{G}$, a  bilinear map \ $\ [\cdot,\cdot]:\mathcal{G}\times \mathcal{G}\rightarrow \mathcal{G}$ \ and a $ \mathbb{K}$-linear map\ $ \alpha:\mathcal{G}\rightarrow \mathcal{G} \ $satisfying
\begin{eqnarray}
&&[x,y]=-[y,x],\\
&&\circlearrowleft_{x,y,z}\big[\alpha(x),[y, z]\big]= 0 \label{jacobie}
\end{eqnarray}
for all $x,\, y,\, z \in \mathcal{G}$
 	($\circlearrowleft_{x, y, z}$ denotes summation
 over the cyclic permutations on $x, y, z$).\\
Let $(\mathcal{G},[\cdot,\cdot],\alpha)$ be a Hom-Lie algebra. A Hom-Lie algebra is called
\begin{itemize}
\item multiplicative  if $\forall x,y\in \mathcal{G}$ we have $\alpha([x,y])=[\alpha(x),\alpha(y)];$
\item regular if $\alpha$ is an automorphism;
\item involutive  if $\alpha$ is an involution, that is $\alpha^{2}=id .$
\end{itemize}
Let $\left( \mathcal{G}, [\cdot, \cdot], ,\alpha \right) $ and $\left(
\mathcal{G}^{\prime }, [\cdot, \cdot] ^{\prime },\alpha^{\prime
}\right) $ be two Hom-Lie algebras. An homomorphism
 $f\ :\mathcal{G}\rightarrow \mathcal{G}^{\prime }$ is said
to be a
\emph{morphism of  Hom-Lie algebras} if
\begin{eqnarray} f ([x, y])&=&[f(x), f(y)] ^{\prime }\quad \forall x,y\in \mathcal{G}
\\ f\circ \alpha&=&\alpha^{\prime
}\circ f.
\end{eqnarray}
\end{defn}
\begin{defn}\cite{Ideal,Ideal1}
An ideal in a Hom-Lie algebra $\mathcal{G}$ is a vector subspace $I$ so that $[\mathcal{G},\alpha(I)] \subset I$. In
other words, $[a,\alpha(x)] \in I$ for all $a\in\mathcal{G}$, $\alpha(x)\in  I$.

A Hom-Lie algebra $\mathcal{G}$ is called abelian if the Lie bracket vanishes for all elements in $\mathcal{G}$:
$[x, y] = 0$
for all $x, y \in\mathcal{G}$.

A non abelian Hom-Lie algebra $\mathcal{G}$ is called simple if it has no non trivial ideals.

Given Hom-Lie algebras $\mathcal{G}$  and $\mathcal{H}$ , their direct sum $\mathcal{G}\oplus \mathcal{H}$   consists of the
vector space direct sum with a bracket operation restricting to the original brackets
on $\mathcal{G}$  and $\mathcal{H}$  and satisfying $[\mathcal{G} , \mathcal{H} ] = 0$.

We define a Hom-Lie algebra $\mathcal{G}$  to be semisimple if it is the finite direct sum of simple Lie
algebras $\mathcal{G}_i$:
$\mathcal{G} = \mathcal{G}_1\oplus\mathcal{G}_2\oplus\dots\oplus\mathcal{G}_n$

Define the lower central series of $\mathcal{G}$ as follows: let  $\mathcal{G}_0=\mathcal{G},\ \mathcal{G}_1=[\mathcal{G},\mathcal{G}],\dots ,\mathcal{G}_n=[\mathcal{G},\mathcal{G}_n]$
. We say $\mathcal{G}$ is nilpotent if $\mathcal{G}_n=\{0\}$ for some $n$.

\end{defn}
\begin{defn}\cite{shengR}
	For any $k\geq -1,$ we call $D\in End(\mathcal{G})$  a $\alpha^{k}$- derivation of the Hom-Lie algebra $(\mathcal{G},[\cdot,\cdot],\alpha)$ if
	\begin{align}
	\alpha \circ D&=D \circ \alpha \quad\text{and }\quad \label{juillet17}\\
	D([x,y])&=[D(x),  \alpha^{k}(y)]+[\alpha^{k}(x),D(y)]\quad   \forall x, y \in \mathcal{G}.\label{13juillet17}
	\end{align}
	We denote by $Der_{\alpha^{k}}(\mathcal{G})$ the set of $\alpha^{k}$-derivations of the Hom-Lie algebra $(\mathcal{G},[\cdot,\cdot],\alpha),$
	and  $$\displaystyle Der(\mathcal{G})=\bigoplus_{k\geq-1}Der_{\alpha^{k}}(\mathcal{G}).$$	
\end{defn}
\begin{defn} \cite{Ayedi}
	A    bilinear form $b$ is called invariant if \[b([x,y],z)=b(x,[y,z])\qquad \forall x,\ y,\ z\in \mathcal{G} .\]  
	
\end{defn}
\subsection{Cohomolgy of Hom-Lie algebras}
First we recall the definition of the representation of multiplicative Hom-Lie algebras.
\begin{defn}(see\cite{shengR})
Let $(\mathcal{G},[\cdot,\cdot],\alpha)$ be a multiplicative Hom-Lie algebra, $V$ be  an arbitrary vector space, and $\beta\in\mathcal{G}l(V)$ be an arbitrary linear self-map on $V$. If  a bilinear map 
\begin{eqnarray*}
[\cdot,\cdot ]_{V}&:&\mathcal{G}\times V \rightarrow V \\
&&(g,v)\mapsto [g,v]_{V}
\end{eqnarray*}
satisfies 
\begin{eqnarray}
[\alpha(x),\beta(v)]_{V}&=&\beta([x,v]_{V}) \label{rep1}\\
\left[[x,y],\beta(v)\right]_{V}&=&\left[\alpha(x),[y,v]\right]_{V}-\left[\alpha(y),[x,v]\right]_{V} .\label{rep2}
\label{mod}
\end{eqnarray}
for all  elements $ x,\, y\in  \mathcal{G}$ and $ v\in V. $\\
Then $(V,[\cdot,\cdot]_{V}, \beta)$  is called a representation of  $\mathcal{G}$, and $ (V,\beta) $ is called a Hom-$\mathcal{G}$-module  via the action $x.v=[x,v]_V$, for all $x\in \mathcal{G}$, $ v\in V. $ 
\end{defn}
\begin{defn}
A $k$-cochain on $\mathcal{G}  $ with values in $V$ is defined to be a $k$-Hom-cochain $f\in C^{k}(\mathcal{G},\ V)$ such it is compatible with $ \alpha $ and $ \beta $ in the sense that $ \beta\circ f =f \circ\alpha $.  Denote $C^{k}_{\alpha,\beta}(\mathcal{G},\ V)  $ the set of $k$-Hom-cochain.

Define $\delta^{k}:C^{k}(\mathcal{G},\ V)\rightarrow C^{k+1}(\mathcal{G},\ V)$ by setting
\begin{align}
&\delta^{k}_r(f)(x_{0},\dots,x_{k})=\nonumber \\&\sum_{0\leq s < t\leq k}(-1)^{t}
f\Big(\alpha(x_{0}),\dots,\alpha(x_{s-1}),[x_{s},x_{t}],\alpha(x_{s+1}),\dots,\widehat{x_{t}},\dots,\alpha(x_{k})\Big) \nonumber \\
&+\sum_{s=0}^{k}(-1)^{s}\Bigg[\alpha^{k+r-1}(x_{s}), f\Big(x_{0},\dots,\widehat{x_{s}},\dots,x_{k}\Big)\Bigg]_{V},\label{def cobbb}
\end{align}
 where $f\in C^{k}(\mathcal{G},\ V)$,  $\ x_{0},....,x_{k}\in \mathcal{G}$ and $\ \widehat{x_{i}}\ $  means that $x_{i}$ is omitted.
 
 For $ k=0 $, for any $a\in \mathcal{G}$, we  take 
 \begin{equation}\label{septembremardi}
 \delta^{0}_r(v)(x)=[\alpha^{r-1}(x),v]_V, \qquad \forall x\in \mathcal{G}, \ v\in V.
 \end{equation}

 \end{defn}

The pair $(\oplus_{k>0}C^{k}_{\alpha ,\beta}(\mathcal{G},\ V),\{\delta^{k}\}_{k>0})$ defines a chomology complex, that is
 $\delta^{k} \circ \delta^{k-1}=0.$
\begin{itemize}
\item The $k$-cocycles space is defined as $Z^{k}(\mathcal{G},V)=\ker \ \delta^{k}.$
 \item  The $k$-coboundary space is defined as    $B^{k}(\mathcal{G},V)=Im \ \delta^{k-1}.$
  \item The $k^{th}$ cohomology  space is the quotient  $H^{k}(\mathcal{G},V)=Z^{k}(\mathcal{G},V)/ B^{k}(\mathcal{G},V). $ It decomposes as well as even and odd $k^{th}$ cohomology spaces.
 \end{itemize}

\section{Symplectic vector spaces}
Let $V$ be a vector space over field $\mathbb{K}$, and let  
$ w:\,V\times V \to  \mathbb{K}$ be a bilinear, skew-symmetric  and non-degenerate form. A symplectic vector space $ (V,w) $ is a vector space with a  bilinear, skew-symmetric, non degenerate form.\\
Let $ (V,w)  $ be a symplectic vector space. Then there is a basis  $ (x_1,\cdots,x_m,y_1,\cdots,y_m) $ such that 
$ w(x_i,x_j) =w(y_i,y_j)=0$ and $ w(x_i,y_j)=\delta_{ij} \quad \forall i,j\in\{1,\cdots,m\}. $ The basis $ (x_1,\cdots,x_m,y_1,\cdots,y_m) $ is called   symplectic basis for $ V. $\\
The symplectic orthogonal complement of some subspace $U\subset V$ is defined as \[U^{\perp}=\{v\in V : w(u,v)=0,\, \forall u\in U\}.\]
We say that $U$  is isotropic if $ U\subset U^{\perp} $  and  Lagrangian if $U=U^{\perp}$.
\begin{prop}
	Let $ U $ be a isotropic subspace of $V$. There exist a isotropic subspace $ W $ such that $ U\cap W=\{ 0\}$,  $ \dim W=\dim U $   and $ U\oplus W $ is symplectic.
\end{prop}

\begin{defn}
A linear map $s: V\to V$ is called a $ \lambda $-symplectic transformation if 
\[w\left( s(u),s(v)\right) =\lambda w\left( u,v\right) ,\qquad \forall\,  u,\, v\in V.\]
A $(2m) \times (2m)$ matrix $S$ with entries in the field $\mathbb{K} $  is said to be $\lambda$-symplectic if
\[\fourIdx{t}{}{}{}{S}\,B\,S=\lambda\, B\] where
\[B=\begin{bmatrix}
0&I_{m}\\
-I_{m}&0\\
\end{bmatrix}.\]
Let  $S=\begin{bmatrix}
X_{m}&Z_{m}\\
T_{m}&Y_{m}\\
\end{bmatrix}$ . Using the identities $ \fourIdx{t}{}{}{}{S}\,B\,S=\lambda\, B, $ by an explicit calculation, we obtain the matrix $S$ is $ \lambda $-symplectic  if and only the three following 
conditions are satisfied:
\begin{equation}
\fourIdx{t}{}{}{}{X_{m}}\,T_{m}\,=\fourIdx{t}{}{}{}{T_{m}}\,X_{m},\ \fourIdx{t}{}{}{}{Z_{m}}\,Y_{m}\,=\fourIdx{t}{}{}{}{Y_{m}}\,Z\, \text{ and  } \fourIdx{t}{}{}{}{X_{m}}\,Y_{m}\,-\fourIdx{t}{}{}{}{T_{m}}\,Z_{m}=\lambda I_{m}.\\
\end{equation}
Let $ Z(s)=\{m\in End(V)/\, m\circ s=s\circ m\}. $ For any $ m\in   Z(s)$ and   $ n\in   Z(s)$, define their commutator $ [m,n] $ as usual : $ [m,n]=m\circ n- n\circ m. $ Then $ (Z(s),[\cdot,\cdot],s) $ is a Hom-Lie algebra. The $ k $-symplectic Hom-Lie algebra $ \mathfrak{sp}_{k}(V,s) $ correspondingly consists of endomorphisms $ f:V\to V $ satisfying \[   w(f(x),s^k(y))+w(s^k(x),f(y))=0 \] for all $ x\in V $ and $ y\in V. $
\end{defn}
\begin{defn}
	Let $(\mathcal{G},[.,.],\alpha)$ be a     Hom-Lie algebra.
	We say that  $ (\mathcal{G},w) $ is a symplectic Hom-Lie algebra  if $ w $ is a $ 2 $-cocycle for the scalar cohomology of $ \mathcal{G}. $  
\end{defn}

\section{Heisenberg Hom-Lie algebras}

\begin{defn} 
 Let $(\mathcal{H},[.,.],\alpha)$ be a     Hom-Lie algebra.
$ \mathcal{H} $ is a Heisenberg type Hom-Lie algebra if and only if $ \mathcal{H} $ is nilpotent and $ dim\left([\mathcal{H},\mathcal{H}]\right)=1. $
\end{defn}

 Let $(\mathcal{H},[.,.],\alpha)$ be a finite-dimensional    Hom-Lie algebra.
Suppose that $  \mathcal{H} $ is a $1$-dimensional  derived ideal \[ D(\mathcal{H})=D_1(H)=[\mathcal{H},\mathcal{H}]=<z>. \]
Then a skew symmetric bilinear form $\mathcal{B}$ can be defined on  $  \mathcal{H} $ by \[ [x,y]=   \mathcal{B}(x,y)z \]
if $ z\notin Z(\mathcal{H}) $, we can deduce 
\[ D_{k+1}(\mathcal{H})=[D_k(\mathcal{H}),\mathcal{H}]=<z> \]

We get the following results.
\begin{prop}
A Hom-Lie algebra $ \mathcal{H} $ is a Heisenberg Hom-Lie algebra only if only it has a $1$-dimensional   derived ideal
 such that \[[\mathcal{H},\mathcal{H}]\subset Z(\mathcal{H})\]
\end{prop}

For the rest of this article, we mean by an Heisenberg  Hom-Lie algebra a multiplicative  Heisenberg  Hom-Lie algebra such that $ [\mathcal{H},\mathcal{H}]= Z(\mathcal{H}) $.
\begin{theorem}
A Heisenberg  Hom-Lie algebra $\mathcal{H}^m$ is a Hom-Lie algebra with $2m+1$  generators  $x_1\dots,x_{m},y_1,\dots,y_m,z  $  with the following structure
\begin{enumerate}[(i)]
  \item The products of the basis elements are given by
\begin{align*}
[x_{k},y_{k}]&=z \qquad k=1,2,\dots,m
\end{align*}
(all other brackets are zero).
\item The matrix of $ \alpha $ with respect to the bases $(x_1\dots,x_{m},y_1,\dots,y_m,z ) $ 
 is of the form
\begin{align}
P=\begin{bmatrix}
X_{mm}&T_{mm}&0_{m1}\\
Z_{mm}&Y_{mm}&0_{m1}\\
L_{1m}&M_{1m}&\lambda
\end{bmatrix}
\end{align}
 where 
$\begin{bmatrix}
X_{mm}&T_{mm}\\
Z_{mm}&Y_{mm}
\end{bmatrix}$
is $\lambda$-symplectic.
\end{enumerate}
\end{theorem}
\begin{cor}
Any multiplicative  Heisenberg Hom-Lie algebra is regular.
\end{cor}
\begin{cor}
	Let $ E=<x_1,\cdots,x_m> $,   $ E^*=<y_1,\dots,y_m>$ and $ V=E\oplus E^* $. Then, 
	$ (V,\mathcal{B})$ is a symplectic vector space,  $ E $ is a lagrangian ideal on $V$ and $ (\mathcal{H}^m,\mathcal{B}) $ is a  symplectic Hom-Lie algebra.
\end{cor}
\begin{prop}
	Let $ (V,w) $ be a symplectic vector space and $ \gamma $ a $ \lambda $-symplectic endomorphism  of  $ V. $ Let $ \alpha $ be an endomorphism of $ V\oplus \C $ defined by $\alpha(v+z)=\gamma(v)+\lambda z\quad \forall (v,z)\in V\times \C $. Then, $ (V\oplus\C,w,\alpha) $ is a Heisenberg Hom-Lie algebra. 
\end{prop}
\begin{example}
Let $\mathcal{H}^m$ be a vector space with basis $(x_1\dots,x_{2m},z ) $. The following triple $(\mathcal{H},\ [.,.],\ \alpha)$ define a Heisenberg  Hom-Lie algebra 
\begin{align*}
[x_{2k-1},x_{2k}]&=z \qquad k=1,2,\dots,m
\end{align*}
(all other brackets are zero).
\[P=matrix\big(\alpha,(x_1\dots,x_{2m},z )\big)=\begin{bmatrix}
P_1&0&\hdotsfor[0.5]{2}&0\\
0&P_2&0&\dots&0\\
\vdots&\ddots&\ddots & \ddots & \vdots \\
.&....&.0&P_m&.0\\
0&\hdotsfor[0.5]{2} &0&\lambda
\end{bmatrix}\]
where
$ P_i\in M_2(\C)$ and $ \ det(P_i)=\lambda.$
\end{example}
\subsection{Heisenberg Hom-Lie algebras of dimension Three}
\begin{prop}
	Let $ \lambda\in\C\setminus\{0,1\} $ and let 
  $(\mathcal{H}^1_{\lambda},\ [.,.],\ \alpha)$  be an Heisenberg  Hom-Lie algebra  over a field $ \C $ of dimension Three. Then, there exist a basis 
  $ (X,Y,Z) $ such that $ [X,Y] =Z,\, [X,Z]=[Y,Z]=0 $ and the matrix of $ \alpha $ has the form $\begin{bmatrix}
  a&c&0\\
  b&d&0\\
  0&0&\lambda
  \end{bmatrix}  $
where $ ad-bc=\lambda. $
\end{prop} 
\begin{example}
From the upper triangular matrices  algebra $ \mathcal{H} $ spanned by
\begin{center}
$X=\left(\begin{array}{ccc}0&1&0\\
0&0&0\\
0&0&0
\end{array}\right) ,$ \
$Y=\left(\begin{array}{ccc}0&0&0\\
0&0&1\\
0&0&0
\end{array}\right), $
$Z=\left(\begin{array}{ccc}0&0&1\\
0&0&0\\
0&0&0
\end{array}\right) ,$
\end{center}
 The defining non zero relation  is $ [X,Y]= Z. $ 
we consider the linear map $\alpha_\lambda:\mathcal{H}\rightarrow \mathcal{H}$ defined by:
\[\alpha_\lambda(X)=aX+bY,\qquad \ \alpha_\lambda(Y)=c X+dY,\qquad \ \alpha_\lambda(Z)=\lambda Z,\]
where $ad-bc=\lambda$. We obtain a family of Heisenberg  Hom-Lie algebra $ (\mathcal{H}^3 ,[\cdot,\cdot],\alpha_\lambda)$.
\end{example}
\begin{example}
With the differential operators $ C^{\infty}(\R^3)\rightarrow C^{\infty}(R^3)$ defined by \[X=\partial_x-\frac{1}{2}y\partial_z,\qquad Y=\partial_y+\frac{1}{2}\partial_z,\qquad Z=\partial_z,\]
and the linear map $\alpha_\lambda:C^{\infty}(\R^3)\rightarrow C^{\infty}(\R^3)$ defined by: \[\alpha_\lambda(X)=a\partial_x+b\partial_y+\frac{1}{2}(b-ay)\partial_z\]
\[\alpha_\lambda(Y)=c\partial_x+d\partial_y+\frac{1}{2}(d-cy)\partial_z\]
\[\alpha_\lambda(Z)=\lambda \partial_z.\]
With $ad-bc=\lambda$,
we can defined  a
three- dimensional Heisenberg  Hom-Lie algebras.
\end{example}
\begin{example}
Consider the algebra of all operators on functions of one real variable. Let $\mathfrak{h}$ denote the three dimensional space of operators spanned by identity operator, $1$, the operator consisting of multiplication by $x$, and the operator $\frac{d}{dx}$ :
\begin{align*}
1&: f\mapsto f\\
X&:f\mapsto xf\\
\frac{d}{dx}&: f\mapsto\frac{d\,f}{dx}.
\end{align*}
Define a skew-symmetric bilinear bracket operation $[\cdot,\cdot]$ on $\mathfrak{h}$ by \[[u,v]=uv-vu,\qquad \forall u, v\in \mathfrak{h}. \]
Define $\alpha\in \mathfrak{g}\mathfrak{l}(\mathfrak{h})$ by \[ \alpha(1)=\lambda 1,\quad \alpha(X)(f)(x)=axf(x)+b\frac{d\,f}{d\,x}(x) ,\quad \alpha(\frac{d}{dx})(f)(x)=cxf(x)+d \frac{d\,f}{d\,x}(x)\]
where $ad-bc=\lambda$.

With the above notations $ (\mathfrak{h},[\cdot,\cdot],\alpha) $ is a  Heisenberg  Hom-Lie algebra.
\end{example}
\begin{prop}
	Let    $(\mathcal{H}^1_{\lambda},\ [.,.],\ \alpha)$  be an Heisenberg  Hom-Lie algebra  over a field $ \C $. Then there exists a basis of $ \mathcal{H}^1_{\lambda} $ such that  $[X,Y] =Z,\, [X,Z]=[Y,Z]=0   $ and The matrix of T with respect to this basis has a one of the following  ones:
\begin{multicols}{2}
	\begin{enumerate}[(i)]
		\item  $\begin{bmatrix}
		1&0&0\\
		1&1&0\\
		0&a&1
		\end{bmatrix}  $ ( $ \lambda=1 $ )
		\item  $\begin{bmatrix}
		1&0&0\\
		0&\lambda&0\\
		0&a&\lambda
		\end{bmatrix}  $
			\item  $\begin{bmatrix}
		\mu&1&0\\
		0&\mu&0\\
		0&0&\mu^2
		\end{bmatrix}  $ ( $ \lambda=\mu^2 $ ). 
					\item  $\begin{bmatrix}
		\mu&0&0\\
		0&\frac{\lambda}{\mu}&0\\
		0&0&\lambda
		\end{bmatrix}  $ . 
	\end{enumerate}
\end{multicols}
\end{prop}
\subsection{Classification of Heisenberg  Hom-Lie algebras.}
 Let $E $ be a vector space over field $ \C $, $ E^{*} $ be its dual space.
Let  $ \gamma : E\rightarrow E $ be an isomorphism. Let $ \beta : E^*\rightarrow E^* $,
$ \mu : E^*\rightarrow E $ and $ \eta :E^*\rightarrow \C$ be   linear transformations.
Define a skew-symmetric bilinear bracket operation $[\cdot,\cdot]$ on $ E\oplus E^*\oplus \C $ by 
\[  [x+f+t,y+g+t']=g(x)-f(y) \qquad \forall x,\, y \in E, f,g\in E^*,\, t,t'\in\C. \]

Define $ \alpha\in End(E\oplus E^*\oplus \C) $ by \[ \alpha(x+f+t)=\gamma(x)+\beta(f)+\mu(f)+\eta(f)+t\lambda \qquad (x,f,t)\in(E,E^*,\C) .\]
\begin{thm}
With the above notations :		
		
	$ (E\oplus E^*\oplus \C,[\cdot,\cdot],\alpha) $ is a Heisenberg  Hom-Lie algebra if and only if  \[\beta\,^{t}\gamma=\lambda I_m \quad\text{ and   } \quad \,^{t}\mu\, \beta=\, ^{t}\beta\, \mu .\] \\
	Then, if $ \mu =0 $ and $ \beta\,^{t}\gamma=\lambda I_m $. The extension of an abelian  Hom-Lie algebra $ (E,[\cdot,\cdot]_0,\gamma) $ by a Hom-module $ (E^*\oplus \C,\beta+\lambda) $:
	$$0\longrightarrow (E^*\oplus\C,\beta+\lambda)\longrightarrow (E\oplus E^*\oplus \C,[\cdot,\cdot],\alpha)\longrightarrow (E,[\cdot,\cdot]_0,\gamma) \longrightarrow 0 $$
	define a Heisenberg  Hom-Lie algebra. We call it the 	Hom-Lie algebra generated by $ E. $
\end{thm}
	\begin{thm}
		Let  $ \lambda\in \C^* $ and $(\mathcal{H}^m_{\lambda},\ [\cdot,\cdot],\ \alpha)$  be an Heisenberg  Hom-Lie algebra over a field $ \C $. Let $ m(x)=(x-\lambda)^{k}(x-\lambda_1)^{k_1}\cdots (x-\lambda_r)^{k_r}$ denote the minimal polynomial of $ \alpha. $ Let $W(\lambda)=\ker(\alpha-\lambda id)^{k}   $  and $ W(\lambda_i)=\ker(\alpha-\lambda_iid)^{k_i} .$
			\begin{enumerate}[(i)]
				\item if $  \lambda_{i}^{2} \neq \lambda $, $  \lambda_{i} \neq \lambda $
	 and $ \lambda_{i} \neq 1 .$ Let $ \mathcal{I}_i=W(\lambda_i)\oplus W(\frac{\lambda}{\lambda_i}). $ Then, $(\mathcal{I}_i \oplus\C z ,\ [.,.],\ \alpha_{/ \mathcal{I}_i \oplus\C z})$  be  Heisenberg  Hom-Lie algebra.
 The  matrix of  restricted 
		transformation $\alpha_{/ \mathcal{I}_i \oplus\C z}:\mathcal{I}_i \oplus\C z \rightarrow \mathcal{I}_i \oplus\C z$
		has  the following forms 
 \begin{align}
			P=\begin{bmatrix}
			X_{k_i,k_i}&0_{k_i,k_i}&0_{k_i,1}\\
			0_{k_i,k_i}&Y_{k_i,k_i}&0_{k_i,1}\\
			0_{1,k_i}&0_{1,k_i}&\lambda
			\end{bmatrix}
			\end{align}
			where 
			$\fourIdx{t}{}{}{}{X_{k_i,k_i}}\,Y_{k_i,k_i}=\lambda I_{k_i,k_i}.$
			\item 	 if $ \lambda \neq 1$	 and $ \lambda_{i}^{2}=\lambda$. We have 
			$(W(\lambda_i) \oplus\C z ,\ [\cdot,\cdot],\ \alpha_{/ W(\lambda_i) \oplus\C z})$  be  Heisenberg  Hom-Lie algebra. The  matrix of  restricted 
			transformation $\alpha_{/ W(\lambda_i)  \oplus\C z}:W(\lambda_i) \oplus\C z \rightarrow W(\lambda_i) \oplus\C z$
			has  the following form
			\begin{align}
			P=\begin{bmatrix}
			X_{k_i,k_i}&T_{k_i,k_i}&0_{k_i,1}\\
			0_{k_i,k_i}&Y_{k_i,k_i}&0_{k_i,1}\\
			0_{1,k_i}&0_{1,k_i}&\lambda
			\end{bmatrix}
			\end{align}
			where 	 $\begin{bmatrix}
			X_{k_i,k_i}&T_{k_i,k_i}\\
			0_{k_i,k_i}&Y_{k_i,k_i}
			\end{bmatrix}$
			is $\lambda$-symplectic 	and $ t_{pq}=0\quad \forall p\neq k_i. $

			\item  if $ \lambda_i=1$	 and $ \lambda \neq 1\ $
			 We have 
			 $(W(\lambda_i) \oplus W(\lambda) ,\ [\cdot,\cdot],\ \alpha_{/ W(\lambda_i) \oplus W(\lambda)})$  be  Heisenberg  Hom-Lie algebra. The  matrix of  restricted 
			 transformation $\alpha_{/ W(\lambda_i)  \oplus W(\lambda)  }:W(\lambda_i) \oplus W(\lambda)  \rightarrow W(\lambda_i) \oplus W(\lambda) $
			 has  the following form
			 \begin{align}
			 P=\begin{bmatrix}
			 X_{k_i,k_i}&0_{k_i,k_i}&0_{k_i,1}\\
			 0_{k_i,k_i}&Y_{k_i,k_i}&0_{k_i,1}\\
			 0_{1,k_i}&M_{1,k_i}&\lambda
			 \end{bmatrix}
			 \end{align}
			 $\begin{bmatrix}
			 X_{k_i,k_i}&0_{k_i,k_i}\\
			 0_{k_i,k_i}&Y_{k_i,k_i}
			 \end{bmatrix}$
			 is $\lambda$-symplectic  and\qquad $ m_{1q}=0\quad \forall q\neq k_i. $
			  	\item  if  $ \lambda = 1\ $ 
			  	$(W(\lambda)  ,\ [\cdot,\cdot],\ \alpha_{/ W(\lambda) })$  be  Heisenberg multiplicative Hom-Lie algebra. The  matrix of  restricted 
			  	transformation $\alpha_{/ W(\lambda)}:W(\lambda)  \rightarrow W(\lambda) $
			  	has  the following form
			  	\begin{align}
			  	P=\begin{bmatrix}
			  	X_{k_i,k_i}&T_{k_i,k_i}&0_{k_i,1}\\
			  	0_{k_i,k_i}&Y_{k_i,k_i}&0_{k_i,1}\\
			  	M_{1,k_i}&0_{1,k_i}&\lambda
			  	\end{bmatrix}
			  	\end{align}
			  	where 	 $\begin{bmatrix}
			  	X_{k_i,k_i}&T_{k_i,k_i}\\
			  	0_{k_i,k_i}&Y_{k_i,k_i}
			  	\end{bmatrix}$
			  	is $\lambda$-symplectic 	and $ t_{pq}=0\quad \forall p\neq k_i. $ 
				\item 	 if $ \lambda = 1$	 and $ \lambda_{i}=-1$. We have 
			$(W(\lambda_i) \oplus\C z ,\ [.,.],\ \alpha_{/ W(\lambda_i) \oplus\C z})$  be  Heisenberg  Hom-Lie algebra. The  matrix of  restricted 
			transformation $\alpha_{/ W(\lambda_i)  \oplus\C z}:W(\lambda_i) \oplus\C z \rightarrow W(\lambda_i) \oplus\C z$
			has  the following form
			\begin{align}
			P=\begin{bmatrix}
			X_{k_i,k_i}&T_{k_i,k_i}&0_{k_i,1}\\
			0_{k_i,k_i}&Y_{k_i,k_i}&0_{k_i,1}\\
			0_{1,k_i}&0_{1,k_i}&\lambda
			\end{bmatrix}
			\end{align}
			where 	 $\begin{bmatrix}
			X_{k_i,k_i}&T_{k_i,k_i}\\
			0_{k_i,k_i}&Y_{k_i,k_i}
			\end{bmatrix}$
			is $\lambda$-symplectic 	and $ t_{pq}=0\quad \forall p\neq k_i. $	
		\end{enumerate}
	\end{thm}
	\begin{proof}
		We have  $ \mathcal{H}^m_{\lambda}=\ker\left(\alpha- \lambda_1I\right)^{k_1} \oplus\cdots \oplus\ker\left(\alpha- \lambda_rI\right)^{k_r}
	\oplus\ker\left(\alpha- \lambda I\right)^{k}	
	 .$
		\begin{Case}
	 $ \lambda_{i}^{2}\neq\lambda $, $ \lambda_{i}\neq\lambda $ and  $ \lambda_{i}\neq 1 $\\
		We choose a
		basis  $  (e^{i}_{1},e_{2}^{i},\cdots,e_{r_{i}}^{i}) $
		corresponding to the  Jordan block  $J(\lambda_{i})$. 
		We have $ [e^{i}_{1},e_{2}^{i}]=\mathcal{B}(e^{i}_{1},e_{2}^{i})z. $ Then  \[B(e^{i}_{1},e_{2}^{i})\lambda z=\alpha(B(e^{i}_{1},e_{2}^{i})z)=\alpha([e^{i}_{1},e_{2}^{i}])=[\alpha(e^{i}_{1}),\alpha(e_{2}^{i})]= [\lambda_ie^{i}_{1},e_{1}^{i}+\lambda_ie_{2}^{i}]=B(e^{i}_{1},e_{2}^{i})\lambda^2_i. \]
		Which gives $[e^{i}_{1},e_{2}^{i}]=0.  $ Then,
		by induction we can show that
		\[ [e^{i}_{k},e_{l}^{i}]=0,\quad \forall k,\ l\in \{1,\cdots,r_i\}.  \]
		Then there exist $k\in \{1,\cdots,r_j\}	$ such that $ [e^{i}_{1},e_{k}^{j}]\neq 0. $\\
		We have \[ B(e^{i}_{1},e_{1}^{j})\lambda z=\alpha([e^{i}_{1},e_{1}^{j}])=[\alpha(e^{i}_{1}),\alpha(e_{1}^{j})]=[\lambda_ie^{i}_{1},\lambda_j e_{1}^{j}]=\lambda_i\lambda_jB(e^{i}_{1},e_{1}^{j}) z.  \]
		 \[ B(e^{i}_{1},e_{k}^{j})\lambda z=\alpha([e^{i}_{1},e_{k}^{j}])
		 =[\alpha(e^{i}_{1}),\alpha(e_{k}^{j})]
		 =[\lambda_ie^{i}_{1},e_{k-1}^{j}+\lambda_j e_{1}^{j}]
		 =\lambda_i\mathcal{B}(e^{i}_{1},e_{k-1}^{j})z
		+ \lambda_i\lambda_jB(e^{i}_{1},e_{1}^{j}) z.  \]
		Hence, by induction, we deduce that $\lambda_i\lambda_j=\lambda.  $	With $\lambda_j\notin \{\lambda,\lambda_i,1\}$, \\
		$ W(\lambda_i)+W(\frac{\lambda}{\lambda_i})+<z> $  is a direct sum. 
	Finally, one can deduce\\
	 $\mathcal{H}_{\lambda}^{k_i}= W(\lambda_i)\oplus W(\frac{\lambda}{\lambda_i}) \oplus<z>  $ is 	
	 Heisenberg  Hom-Lie algebra. The  matrix of  restricted 
	transformation $\alpha_{/\mathcal{H}_{\lambda}^{k_i}}:\mathcal{H}_{\lambda}^{k_i} \rightarrow  \mathcal{H}_{\lambda}^{k_i}$
	has  the following form
	\begin{align}
	\begin{bmatrix}
	X_{k_i,k_i}&0_{k_i,k_i}&0_{k_i,1}\\
	0_{k_i,k_i}&Y_{k_i,k_i}&0_{k_i,1}\\
	0_{1,k_i}&0_{1,k_i}&\lambda
	\end{bmatrix}
	\end{align}	
		where 
	$\fourIdx{t}{}{}{}{X_{k_i,k_i}}\,Y_{k_i,k_i}=\lambda I$.		 			
	\end{Case}	
\begin{Case}
	$ \lambda\neq 1 $ and $ \lambda_{i}^{2}=\lambda $.\\
	We choose a
	basis  $  (e^{i}_{1},e_{2}^{i},\cdots,e_{2r_{i}}^{i}) $
	corresponding to the  Jordan block  $J(\lambda_{i})$.\\
	We suppose that $ [e^{i}_{1},e^{i}_{1}]= [e^{i}_{1},e_{2}^{i}] =\cdots=[e^{i}_{1},e_{k-1}^{i}]=0$ and $ [e^{i}_{1},e_{k}^{i}]\neq 0 .$ Let $ [e^{i}_{1},e_{k+1}^{i}]=\mu_{k+1,i}z .$
	We have 
	\begin{align*}
	\mu_{k+1,i}\lambda z=\alpha(\mu_{k+1,i}z)&= \alpha\left( [e^{i}_{1},e_{k+1}^{i}]\right)\\
	& = [\alpha(e^{i}_{1}),\alpha(e_{k+1}^{i})]=[\lambda_ie^{i}_{1},e_{k}^{i}+\lambda_ie_{k+1}^{i}]=
	\lambda_i[e^{i}_{1},e_{k}^{i}]+\lambda_i^2\mu_{k+1,i}z.
	\end{align*}
	We obtain $\lambda_{i}[e^{i}_{1},e_{k}^{i}]=0.  $ That it is not true. We deduce 
	\[[e^{i}_{1},e_{2}^{i}] =\cdots=[e^{i},e_{2r_{i}-1}^{i}]=0 \quad \text{ and }\quad [e^{i}_1,e_{2r_{i}}^{i}]\neq 0. \] 
	Then, by induction we obtain
	\[ [e^i_{k},e^i_{1}]=[e^{i}_{k},e^i_2]=\cdots=
	[e^i_k,e_{2r_{i}-k}^i]=0,\quad \forall k\in\{1,\cdots,r_i\}  \] and $[e^i_k,e^i_{2r_{i}-k+1}]\neq 0.  $ Denote $ V_i=<e^{i}_{1},\cdots,e^{i}_{r_i}> .$ Then, the subspace $ V_i  $ is  isotropic. Denote $ W_i=<e_{r_i+1},\cdots,e_{2r_i}> $ and $ \mathcal{I}_i=V_i\oplus W_i. $\\
	We will prove by induction that, for all  $k \in \{1,\cdots,r_i\} $ we can write
	\begin{equation}\label{defH}
	V_i=<u_1,u_2,\cdots,u_k>\oplus H_k  \text{ and } 
	W_i= L_k\oplus<w_{2r_i-k+1},w_{2r_i-k},\cdots,w_{2r_i-1},w_{2r_i} > 
	\end{equation}
	such that
	\begin{align*}
	&[u_p,u_q]=[u_p,H_k]=[u_p,L_k]=[u_p,L_k]=[u_p,w_{l}]=0,\quad \forall l\neq 2r_{i}-p+1\\
	&[w_p,w_q]=[w_p,H_k]=[w_p,L_k]=0,\\
	&[u_{p},w_{2r_i-p+1}]=z,\quad \forall p\in\{1,\cdots,k\}.
	\end{align*}
	\textbf{Base case}: When $ k=1 $, denote $ u_1=\frac{1}{\mu_{1,2r_i}}e^{i}_{1} $ where $ [e^{i}_1,e_{2r_{i}}^{i}]=\mu_{1,2r_i}z $
	and $ w_{i}=e^i_{2r_i} .$	
	Recall that $ [x,y]=B(x,y)z. $ We define a linear functional $ g_{2r_i}:V_i\longrightarrow \R $
	by $ g_{2r_i}(x)=\mathcal{B}(w_{i},x) .$ We have $ V_i=<u_1>\oplus H_1 $ where $ H_1=\ker  g_{2r_i}.  $
	We define a linear functional $ f_{2r_i}:\mathcal{I}_i\longrightarrow \R $
	by $ f_{2r_i}(x)=\mathcal{B}(w_{_i},x) .$ We have $ \mathcal{I}_i=<u_1>\oplus H_1\oplus H'_1\oplus<w_{r_i}> $ where $ H_1'\oplus <w_{r_i}>=\ker  f_{2r_i}.  $
	We define a linear functional \[ h_{1}:<u_1>\oplus H'_1\oplus<w_{r_i}>\longrightarrow \R 
	\quad \text{ by }\quad h_{1}(x)=\mathcal{B}(u_{1},x) .\] We have $<u_1>\oplus H'_1\oplus<w_i>
	=<u_1>\oplus L_1 \oplus<w_{i}>$ where \[<u_1>\oplus L_1
	=\ker  h_{1}.  \]
	We deduce $ \mathcal{I}_i =<u_1>\oplus H_1\oplus L_1\oplus<w_{r_i} > $ such that \[ [u_1,H_1]=[u_1,L_1]=[w_i,H_1]=[w_i,L_1]=0, \quad [u_1,w_{r_i}]=z ,\] 
	\textbf{Induction step} : Let $ k\in\{1,\cdots,r_{i}-1\} $
	be given and suppose \eqref{defH} is true for $ k $. 
	Let $ u_{k+1}\in H_k $. Since $ 0=[u_{k+1},u_p]=[u_{k+1},H_k]=[u_{k+1},w_q] $, by $ Z(\mathcal{H}^m_{\lambda})=<z> $ there exist
	$ w_{2r_i-k}\in  L_k$ such that $ B(u_k,w_{2r_{i}-k})\neq 0. $
	With the same method as in the previous case, we obtain
	\begin{equation*}
	V_i=<u_1,u_2,\cdots,u_{k+1}>\oplus H_{k+1}  \text{ and } 
	W_i= L_{k+1}\oplus<w_{2r_i-k},w_{2r_i-k+1},\cdots,w_{2r_i-1},w_{2r_i} > 
	\end{equation*}
	such that
	\begin{align*}
	&[u_p,u_q]=[u_p,H_{k+1}]=[u_p,L_{k+1}]=[u_p,L_{k+1}]=[u_p,w_{l}]=0,\quad \forall l\neq 2r_{i}-p+1\\
	&[w_p,w_q]=[w_p,H_{k+1}]=[w_p,L_{k+1}]=0,\\
	&[u_{p},w_{2r_i-p+1}]=z,\quad \forall p\in\{1,\cdots,k+1\}.
	\end{align*}	 
	\textbf{Conclusion}: By the principle of induction, \eqref{defH} is true for all $k \in \{1,\cdots,r_i\} $.	\\
	Thus, with $ k= r_i$ ,  we obtain  $W(\lambda_i) \oplus\C z=\mathcal{H}^{r_i}_{\lambda}$. The matrix representation of 
	$\alpha_{/ \mathcal{H}^{r_i}_{\lambda}}  $	 respect to the basis $ (u_1,u_2,\cdots,u_{r_i},w_{1},\cdots,w_{r_i},z) $ is of the form
	\begin{align}
	\begin{bmatrix}
	X_{r_i,r_i}&T_{r_i,r_i}&0_{r_i,1}\\
	0_{r_i,r_i}&Y_{r_i,r_i}&0_{r_i,1}\\
	0_{1,r_i}&0_{1,r_i}&\lambda
	\end{bmatrix}
	\end{align}
	where 
	$\fourIdx{t}{}{}{}{X_{r_i,r_i}}\,Y_{r_i,r_i}=\lambda\, I$, \quad 
	$\fourIdx{t}{}{}{}{Y_{r_i,r_i}}\,T_{r_i,r_i}
	=\fourIdx{t}{}{}{}{T_{r_i,r_i}}\,Y_{r_i,r_i}$	 
	and $ t_{pq}=0\quad \forall p\neq r_i. $
\end{Case}	 	
The proof of (\crm{3}) , (\crm{4}) ,(\crm{5})   is very similar to the proof of (\crm{1}) and (\crm{2}) ,

	
		
		\end{proof}
\begin{cor}\label{hm} If $ \lambda\neq 1 $ or $ k=1 $.
There is a basis of $\mathcal{H}^m_{\lambda}  $ given by $ (x_1,\cdots,x_m, y_1,\cdots,y_m,z)$ such that $ [x_{i},y_{j}]=\delta_{ij}z,[x_i,y_j]=[y_i,y_j]=0 $ and the corresponding  matrix is of form
\begin{align}\label{alpha}
			P=\begin{bmatrix}
			X_{m,m}&T_{m,m}&0_{m,1}\\
			0_{m,m}&Y_{m,m}&0_{m,1}\\
			0_{1,m}&M_{1,m}&\lambda
			\end{bmatrix}
			\end{align}
Where $\fourIdx{t}{}{}{}{X}\,Y=\lambda\, I$ and $\fourIdx{t}{}{}{}{T}\,Y=\fourIdx{t}{}{}{}{Y}\,T.$
\end{cor}
\begin{remq}
	The previous corollary can be written differently :\\
	Any Heisenberg  Hom-Lie algebra $ \mathcal{H}^m_{\lambda} $ such that $ \lambda\neq 1 $ or $ k=1 $ contains a Lagrangian subspace $ E $ such that $\mathcal{H}^m_{\lambda}= E\oplus E^*\oplus <z> $ and $ \alpha(E)=E .$
\end{remq}
We
summarize the main facts in the theorem below.
\begin{thm} Let 
	Let $(\mathcal{H}^m_{\lambda},\ [\cdot,\cdot],\ \alpha)$  be an Heisenberg  Hom-Lie algebra over a field $ \C $ such that
$ \alpha(E)=E $	
	. Then, it is equivalent to one and only one on the Heisenberg  Hom-Lie algebra given  by one of the following extensions 
\begin{enumerate}
	\item \[0\longrightarrow (\C,\lambda)\longrightarrow (E\oplus E^*\oplus \C,[\cdot,\cdot],\alpha)\longrightarrow (E\oplus E^*,[\cdot,\cdot]_0,\gamma+\lambda\, ^{t}\gamma^{-1}) \longrightarrow 0 \]
	where $ [x+f+t,y+g]+t'=g(x)-f(y) .$ The matrix of $ \alpha $ relative to the bases $ (e_1,\cdots,e_m,e_1^*,\cdots,e_m^*,1) $ was of the form
\begin{equation}\label{ext}
 \begin{bmatrix}
X_{m,m}&0_{m,m}&0_{m,1}\\
0_{m,m}&Y_{m,m}&0_{m,1}\\
0_{1,m}&0_{1,m}&\lambda
\end{bmatrix}.
\end{equation}
Where $\fourIdx{t}{}{}{}{X_{m,m}}\,Y_{m,m}=\lambda\, I_{m}$.
		\item \[0\longrightarrow (\C,\lambda)\longrightarrow (E\oplus E^*\oplus \C,[\cdot,\cdot],\alpha)\longrightarrow (E\oplus E^*,[\cdot,\cdot]_0,\gamma+\lambda\, ^{t}\gamma^{-1}+\mu) \longrightarrow 0 \]
	where $ [x+f+t,y+g]+t'=g(x)-f(y) ,$  and $ \gamma\, ^{t}\mu=\mu\,^{t}\gamma $. The matrix of $ \alpha $ relative to the bases $ (e_1,\cdots,e_m,e_1^*,\cdots,e_m^*,1) $ was of the form
	$ \begin{bmatrix}
	X_{m,m}&T_{m,m}&0_{m,1}\\
	0_{m,m}&Y_{m,m}&0_{m,1}\\
	0_{1,m}&0_{1,m}&\lambda
	\end{bmatrix} $.Where $\fourIdx{t}{}{}{}{X_{m,m}}\,Y_{m,m}=\lambda\, I_{m}$ and
	$\fourIdx{t}{}{}{}{T_{m,m}}\,Y_{m,m}=\fourIdx{t}{}{}{}{Y_{m,m}}\,T_{m,m}.$ 
	\item \[0\longrightarrow (E^*\oplus\C,\lambda)\longrightarrow (E\oplus E^*\oplus \C,[\cdot,\cdot],\alpha)\longrightarrow (E,[\cdot,\cdot]_0,\gamma) \longrightarrow 0 \]
where $ [x+f+t,y+g]+t'=g(x)-f(y) .$ The matrix of $ \alpha $ relative to the bases $ (e_1,\cdots,e_m,e_1^*,\cdots,e_m^*,1) $ was of the form
$ \begin{bmatrix}
X_{m,m}&0_{m,m}&0_{m,1}\\
0_{m,m}&Y_{m,m}&0_{m,1}\\
0_{1,m}&M_{1,m}&\lambda
\end{bmatrix} $. Where $\fourIdx{t}{}{}{}{X_{m,m}}\,Y_{m,m}=\lambda\, I_{m}$.
\end{enumerate}	
\end{thm}
\begin{prop}
 Let   $ \mathcal{H} $ be a finite dimensional Hom-Lie algebra  with a $1$-dimensional  derived ideal. Suppose that
$[\mathcal{G}, \mathcal{G}] = <h>\subset Z(\mathcal{G}).$
Then, there exists an abelian Hom-Lie sub-algebra $ \mathfrak{a} $ such that \[\mathcal{H}=\mathcal{H}^m_{\lambda}  \oplus \mathfrak{a}.\] Where $\mathcal{H}^m_{\lambda} $ is
the Heisenberg Hom-Lie algebra defined in Corollary \ref{hm}.
\end{prop}
\subsection{Oscillator Hom-algebra}
The Oscillator algebra (see\cite{oscillator}) $ \mathcal{H} $ is spanned by
 $\{a_n : n \in \Z \}\cup \{\hbar\}$, where $\hbar$ is central
and $ [a_m,a_n]=m\delta_{m+n,0}h $.
The oscillator algebra has a triangular decomposition $ \mathcal{H}=n^-\oplus \mathfrak{h}\oplus n^+ $ where 
\[\displaystyle  n^-=\bigoplus_{k=1}^{+\infty}\C a_{-k}, \qquad 
\displaystyle  \mathfrak{h}=\C \hbar \oplus \C a_0,\qquad
 \displaystyle  n^+=\bigoplus_{k=1}^{+\infty}\C a_{k}. \]
 \begin{defn}
 	A oscillator algebra $ \mathcal{H} $ is called Hom-oscillator algebra if there exist a linear map $ \alpha: \mathcal{H}\to \mathcal{H}$ satisfying 
 	\[\alpha([x,y])=[\alpha(x),\alpha(y)] ,\qquad \forall x,y\in \mathcal{H} . \]
 \end{defn}

\begin{example}
	Let $ \mathcal{A}=\C_1[x_1,x_2,\cdots]$ be the vector space of homogeneous polynomials of degree $ 1 $ in $x_1,x_2,\cdots  $ and $ \mathcal{B}=\C_{1}[\frac{\partial}{\partial x_1},\frac{\partial}{\partial x_2},\cdots] $ be the vector space of homogeneous polynomials of degree $ 1 $ in
$	\frac{\partial}{\partial x_1},\frac{\partial}{\partial x_2},\cdots$. Define a skew-symmetric bilinear bracket operation $ [\cdot,\cdot] $ on $ \mathcal{A}\oplus  \mathcal{B}$ by 
\begin{align*}
[P+\frac{\partial}{\partial x_i},Q+\frac{\partial}{\partial x_j}]&=\frac{\partial Q}{\partial x_i} -  \frac{\partial P}{\partial x_j},\\
[P,id_{\mathcal{A}}]=[id_{\mathcal{A}}, P ] =[id_{\mathcal{A}},\frac{\partial}{\partial x_i}]&=[\frac{\partial}{\partial x_i},id_{\mathcal{A}}]=0,\quad \forall P,\,Q\in  \mathcal{A},\quad i,\, j\in \N^* . \end{align*}  
  Then, $ \mathcal{H} =\mathcal{A} \oplus \mathcal{B}  $ is 
  a 	oscillator algebra.\\   
 Denote by 
	\[\hbar=1 ,\quad a_0=id_{\mathcal{A}},\quad a_n=\frac{\partial}{\partial x_n},\quad a_{-n}=nx_n,\quad \forall n\in \N^*.\]
Let $ \gamma:  n^- \to  n^- $, $ \beta:  n^+ \to  n^+ $be two linear maps. 
Let $ \alpha $ be a linear transformation from $  \mathcal{H} $ into it self such that \[\alpha(a_{-n})=\gamma(a_{-n})\quad \text{ and } \quad  \alpha(a_{n})=\beta(a_{n}),\quad \forall n\in \N^*. \]
$ \mathcal{H} $ is a Hom-oscillator algebra only if only $ \gamma $ is bijective, \[\beta(a_n)=\lambda ^t\gamma^{-1}(a_n) =\lambda  \frac{\partial}{\partial x_n}\circ \gamma^{-1}, \, \forall n\in \N^*\]
and there exist $ \lambda\in\C^* $ such that $ \alpha(\hbar)=\lambda \hbar $.
\end{example}
\section{Derivation of Heisenberg Hom-Lie algebras}

Let   $(\mathcal{H}^m_{\lambda},\ [\cdot,\cdot],\ \alpha)$  be  Heisenberg  Hom-Lie algebra generated by a vector space $ V $. There is a basis of $\mathcal{H}^m  $ given by $ (x_1,\cdots,x_m, y_1,\cdots,y_m,z)$ such that \[ [x_{i},y_{j}]=\delta_{ij}z,\quad     [x_i,x_j]=[y_i,y_j]=[x_i,z]=[y_i,z]=0 \] and the corresponding  matrix is of form
\[P=\begin{bmatrix}
X_{m}&0&0\\
0&Y_{m}&0\\
0&0&\lambda
\end{bmatrix},\]
where $\fourIdx{t}{}{}{}{X_{m}}\,Y_{m}=\lambda\, I_{m}.$
Let $ \lambda_1,\cdots,\lambda_r $ be the distinct eigenvalues of $ X $
 corresponding to the multiplicities $m_1,\cdots,m_r  $.
\begin{prop}
	Let $ A\in End(V) $. Then $ A\in  \mathfrak{sp}_{k} (V)$ if and only if \[   A=S^{k}M=\begin{bmatrix}
	X_{m}^{k}&0_{m}\\
	0_{m}&\lambda^{k}\fourIdx{t}{}{}{}{X^{-k}_{m}}
	\end{bmatrix}\times\begin{bmatrix}
	U_{m}&W_{m}\\
	V_{m}&-\fourIdx{t}{}{}{}{U_{m}}
	\end{bmatrix} \]
	where the off-diagonal blocks $ 	V_{m} $ and $ W_{m} $ of $ M $ are symmetric.
\end{prop}
\begin{theorem}
Let $D\in End(\mathcal{H}^m_{\lambda})$. Then $D$ is a $ \alpha^k $-derivation if and only if
\[D=\begin{bmatrix}
D_1&D_3&0\\
D_2&D_4&0\\
U_{1,m}&V_{1,m}&\mu
\end{bmatrix},\]
where 
\begin{align*}
& D_1X=XD_1, \quad  
D_4=\mu  \, \fourIdx{t}{}{}{}{ X^{-k}}-
\lambda^k\, \fourIdx{t}{}{}{}{(D_1X^{-2k})} 
,\\
& \fourIdx{t}{}{}{}{X}\,D_2X=\lambda D_2,\quad
 \fourIdx{t}{}{}{}{X^k}\,D_2=\fourIdx{t}{}{}{}{D_2}\,X^k,\quad\\
& XD_3\fourIdx{t}{}{}{}{X}=\lambda D_3 \quad X^k\, \fourIdx{t}{}{}{}{D_3}=D_3\fourIdx{t}{}{}{}{X^k},
   \\
  & \fourIdx{t}{}{}{}{X}\, \fourIdx{t}{}{}{}{U_{1,m}}=\lambda \, \fourIdx{t}{}{}{}{U_{1,m}} \text{\quad  and\quad }
   X \, \fourIdx{t}{}{}{}{V_{1,m}}=\fourIdx{t}{}{}{}{V_{1,m}}.
\end{align*}
Moreover, for each $ D\in Der_{\alpha^{k}} $ there is an endomorphism $    A\in \mathfrak{sp}_{k} (V,\alpha_{/V})$ 
such that
\[ D=\begin{bmatrix}
aS+A&0\\
U_{1,2m}&2a
\end{bmatrix} \]

\end{theorem}
\begin{proof}
With \eqref{juillet17}, we obtain $D_1X=XD_1  $, $ \fourIdx{t}{}{}{}{X}\,D_2X=\lambda D_2, $\\	
$ \fourIdx{t}{}{}{}{X}\, \fourIdx{t}{}{}{}{U_{1,m}}=\lambda \, \fourIdx{t}{}{}{}{U_{1,m}}, $ $ XD_3\fourIdx{t}{}{}{}{X}=\lambda D_3 , $ $ \fourIdx{t}{}{}{}{X}\,D_4= D_4\, \fourIdx{t}{}{}{}{X},$
and $   X \, \fourIdx{t}{}{}{}{V_{1,m}}=\fourIdx{t}{}{}{}{V_{1,m}}. $\\
With \eqref{13juillet17}, we obtain 
$\mu\mathcal{B}(x,y)=\mathcal{B}\left( D(x),\alpha^{k}(y)\right) +\mathcal{B}\left( \alpha^{k}(x),D(y)\right)  .$\\
Then, $ \mu B=\fourIdx{t}{}{}{}{D}BS^k+\fourIdx{t}{}{}{}{S^k}BD. $
Thus, by a matrix calculation, one obtains the equalities remaining.
\end{proof}
\begin{prop}
	\begin{enumerate}[(i)]
		\item 	If $ \lambda_i \lambda_j\neq \lambda,\ \forall i,j\in\{1,\cdots,r\} $. Then, $ D_2=D_3=0. $ 
		\item  	If $ \lambda_i \neq \lambda,\ \forall i\in\{1,\cdots,r\} $. Then, $ U_{1m}=0. $ 
			\item  	If $ \lambda_i \neq 1,\ \forall i\in\{1,\cdots,r\} $. Then, $ V_{1m}=0. $ 
	\end{enumerate}
\end{prop}
\begin{lem}
	   Let $\displaystyle \chi(x)=\prod_{i=1}^{r}(x-\lambda_i)^{m_i} $ be the characteristic polynomial of  $ X $. 
	   The eigenspace of $ X $ corresponding to $ \lambda_i $, is
	  denoted by $ E_{\lambda_i} $. Then \[ \displaystyle \dim C^{1}_{\alpha,\alpha}(E,E) =\sum_{k=1}^{r}m_{k}\dim E_{\lambda_k}.\]
\end{lem}
\begin{theorem}
Let $ I=\{\lambda_i;\, \chi(\frac{\lambda}{\lambda_i})=0 \} $. Then, 
\begin{enumerate}[(i)]
	\item If $ \chi(1)\neq0 $ and $ \chi(\lambda)\neq 0. $
	\[ \dim(Der_{\alpha^k})=card(I)+\sum_{i=1}^{r}m_{i}\dim E_{\lambda_i}+1 \]
	\item  If $ \chi(1)=0 $ and $ \chi(\lambda)\neq 0. $
	\[ \dim(Der_{\alpha^k})=card(I)+\sum_{i=1}^{r}m_{i}\dim E_{\lambda_i}+\dim(E_1)+ 1 \]
		\item  If $ \chi(1)\neq 0 $ and $ \chi(\lambda)= 0. $
	\[ \dim(Der_{\alpha^k})=card(I)+\sum_{i=1}^{r}m_{i}\dim E_{\lambda_i}+\dim(E_\lambda)+ 1 \]
			\item  If $ \chi(1)= 0 $ and $ \chi(\lambda)= 0. $
	\[ \dim(Der_{\alpha^k})=card(I)+\sum_{i=1}^{r}m_{i}\dim E_{\lambda_i}+\dim(E_1)+\dim(E_\lambda)+ 1 \]
\end{enumerate}	
\end{theorem}
In the following table, we give all Heisenberg Hom-Lie algebras algebras of dimension $3$, we determine the space of the derivations, their dimensions:\\

\begin{tabular}{|c|c|c|}
	\hline
	Matrix $ S $ & Derivation D& Dimension \\
	
	\hline

	$\begin{bmatrix}
	a&0&0\\
	0&\frac{\lambda}{a}&0\\
	0&0&\lambda
	\end{bmatrix},$
	where  $ a^2\neq \lambda$ and $a\neq \lambda $& 	$\begin{bmatrix}
	d_1&0&0\\
	0&d_4&0\\
	0&0&d_4a^k+d_1(\frac{\lambda}{a})^k
	\end{bmatrix},$ &$\dim(Der_{\alpha^k})= 2$\\

	\hline
	    $ \begin{bmatrix}
		a&0&0\\
		0&a&0\\
		0&0&\lambda
	\end{bmatrix},$
	where  $ a^2= \lambda$ and $a\neq \lambda $  &
	 $\begin{bmatrix}
	 d_1&d_3&0\\
	 d_2&d_4&0\\
	 0&0&a^k(d_4+d_1)
	 \end{bmatrix},$
	 &$\dim(Der_{\alpha^k})= 4$  \\
	\hline
		\hline
	$ \begin{bmatrix}
	1&0&0\\
	0&1&0\\
	0&0&1
	\end{bmatrix},$
  &
	$\begin{bmatrix}
	d_1&-d_2&0\\
	d_2&d_1&0\\
	0&0&2d_1
	\end{bmatrix},$
	&$\dim(Der_{\alpha^k})= 2$  \\
	\hline
	$ \begin{bmatrix}
	a&1&0\\
	0&a&0\\
	0&0&\lambda
	\end{bmatrix},$ where  $ a^2=\lambda $ and $ a\neq 1 $
	&
	$\begin{bmatrix}
	d_1&d_3&0\\
	0&d_1&0\\
	0&0&2a^4d_1
	\end{bmatrix},$
	&$\dim(Der_{\alpha^k})= 2$  \\
	\hline
$ \begin{bmatrix}
1&1&0\\
0&1&0\\
0&0&1
\end{bmatrix},$  
&
$\begin{bmatrix}
d_1&d_3&0\\
d_2&d_4&0\\
u&v&d_1+d_4-kd_2
\end{bmatrix},$
&$\dim(Der_{\alpha^k})= 6$  \\
\hline
\end{tabular}
\subsection{The Hom-Lie algebra $  \mathcal{H}(D) $}
In this subsection we will present some results similar to those contained in \cite{super}.
Given a $ \alpha $-derivation $D$ of a Hom-Lie algebra $ \mathcal{G} $, we may consider the vector space $  \mathcal{G}(D)=\mathcal{G}\oplus \C D$ and endow it with the Hom-Lie algebra structure defined by \[[x+a D,y+b D]_D=[x,y]+a D(y)-b D(x)\] 
 and $\gamma_D(x,a D)=(\alpha(x),a D)$ for all $x,\ y\in \mathcal{G}, a, b \in \C$.(see \cite{shengR}).\\
If $ \mathcal{G} $ is nilpotent, the Hom-Lie algebra $  \mathcal{G}(D) $ is nilpotent if and only if $D$ is a nilpotent derivation. Then, if  $  \mathcal{H}_{\lambda}^{m}(D) $ is nilpotent, $ D(z)=0. $
 \begin{theorem}
 Let $ D,\, D' $ be two $ \alpha $-derivations in $ \mathcal{H}^m_{\lambda} $. The   Hom-Lie algebras   $ \mathcal{H}^m_{\lambda}(D) $   and  $ \mathcal{H}^m_{\lambda}(D') $ are isomorphic if, and only if   $ D'=\frac{1}{a}\left(\varphi\circ D\circ\varphi^{-1}-ad_{v}\right)$ where $ a\in\C^*$, $ \varphi\in Aut(\mathcal{H}^m_{\lambda}) $ and $v\in \mathcal{H}^m_{\lambda}$. 
 \end{theorem}
\begin{thm}
Let $b $ be a symmetric, non-degenerate,  bilinear form. $ b $ 
is invariant if and only if  $  \ker(D)=\{z\}. $
 In this case, $ b:\mathcal{H}^m_{\lambda}(D)\times\mathcal{H}^m_{\lambda}(D) \to \C$  is given by \[b= \begin{bmatrix}
 	\gamma_{1}(\frac{1}{a}\fourIdx{t}{}{}{}{S^{-1}}+\fourIdx{t}{}{}{}{A^{-1}})w&0_{2m,1}&0_{2m,1}\\
 	0_{1,2m}&0_{1,2m}&\gamma_{1}\\
 	0_{1,2m}& \gamma_{1}&\gamma_{2}
 \end{bmatrix}.\]
 Where $ \gamma_{1}=b(z,D) $ and  $ \gamma_{2}=b(D,D) $.
\end{thm}
\subsection{The meta Heisenberg Hom algebra }
Let $ I_{k} $,  $ I^{*}_{k} $ be the linear transformations of $ \mathcal{H}^m_{\lambda} $ by setting 
\begin{align*}
	I_{k}(v+f+x)&=\alpha^{k}(v)+\lambda^{k}x\\
		I_{k}^{*}(v+f+x)&=\alpha^{k}(f)+\lambda^{k}x
\end{align*} 
$ \forall v\in E,\, f\in E^{*}, \, x\in \C.$ Then $ I_{k},  \, I^{*}_{k}\in  Der_{\alpha^{k}}(\mathcal{H}^m_{\lambda})$.\\
Let $ m_{k}(E)=\mathcal{H}^m_{\lambda}\oplus \C I_{k}\oplus\C  	I_{k}^{*}.$ Then $ m_{k}(E) $ is a Hom-Lie algebra which is a extension of abelian  Hom-Lie algebra $\C I_{k}\oplus\C  	I_{k}^{*}  $ by the Heisenberg algebra  $ \mathcal{H}^m_{\lambda} $. We call it  meta Heisenberg Hom  algebra generated by $ E $ (for the meta Heisenberg algebra see \cite{Meta}).
\begin{prop}
	A bilinear  form $ b:m_{k}(E)\times m_{k}(E) \to \C   $ is invariant  if and only if $ b $ is trivial.
\end{prop}
\begin{prop}
	Let $(s,[\cdot,\cdot]_{s},id_s)  $ be a semisimple Hom-Lie algebra. $ E $ be a Hom-$ s $-module. The following results hold.
	\begin{enumerate}[(i)]
		\item The  meta Heisenberg Hom  algebra generated by $ E $ can become a Hom-$ s $-module satisfying 
		\begin{align*}
		a(e+f+xc+yI+zI^*)&=ae+af   \qquad \forall a\in s,\, e\in E,\, f\in E^*,\, x,\, y,\, z\in\C
		\end{align*}
		where $ af\in E^* $ defined by $ af(\alpha^ k(e))=-\alpha^k(f)(ae) $ and $ s $ acts on $ m_{k}(E) $ as derivation of $ m_{k}(E) $.
		\item There is a extension $ \mathcal{G} $ of $ s $ by $ m_{k}(E) $. The radical of $ \mathcal{G} $ is $ m_{k}(E)  $. The nilpotent radical is $ \mathcal{H}^m_{\lambda}. $
	\end{enumerate}
\end{prop}

\section{The  faithful representation for Heisenberg Hom-Lie algebra}
The faithful representations of Lie superalgebras are studied in \cite{faithfull,Minimalfaithfull}. In this section we define and study  the  faithful representation for Heisenberg Hom-Lie algebra.
\begin{defn}
A representation $(V,[\cdot,\cdot]_{V}, \beta)$ is faithful if $ \beta $ is  an isomorphism  and the map
\begin{center}
\begin{tabular}{lll}
$\rho\colon$  &$\mathcal{G}$&$\longrightarrow  End(V)$\\
           & $x$&$\longmapsto [x,\cdot]_V$
\end{tabular}
\end{center}
 is injective.
\end{defn}
\begin{prop}
Let $(V,[\cdot,\cdot]_{V}, \beta)$ be a representation of  $\mathcal{H}^m_{\lambda}$. We suppose $ \beta $ is an isomorphim of $V$.
Then  $ \rho $ is a faithful representation if $ \rho(z)  \neq0. $
\end{prop}
\begin{prop}
If $(V,[\cdot,\cdot]_{V}, \beta)$ is  a irreducible representation of  $\mathcal{H}^m_{\lambda}$ satisfying $ \beta\left(\rho(z)\right)  \neq0. $  then  $ V $ is  is a faithful representation for $\mathcal{H}^m_{\lambda}$.
\end{prop}
\subsection{The minimal faithful representation for Heisenberg Lie algebra}
Let $\mathcal{G}$ be a Hom-Lie algebra and write \[\mu(\mathcal{G})=min\{dim\, V/\, V\textit{ is a faithful }  \mathcal{G} \textit{-module}\}.\]
\begin{theorem}
We have 
\[\mu(\mathcal{H}^m_\lambda)=m+2.\]

\end{theorem}
Let $ (v_1,\cdots,v_{m+2}) $ be  a basis for a vector space $ F $. Consider the bilinear mapping \[[\cdot,\cdot]_F:\mathcal{H}^m_\lambda\times F\rightarrow F\] given by 
\begin{align*}
[x_i,v_j]_F&=\delta_{i,j}v_{m+1}\\
[y_i,v_j]_F&=\delta_{j,m+2}\beta(v_i)\\
[z,v_j]_F&=\delta_{j,m+2}\lambda v_{m+1}.\\
\end{align*}
Define $ \beta  \in End(F) $ by 
\begin{align*}
\begin{bmatrix}
Y_{m,m}&0_{m,1}&0_{m,1}\\
0_{1,m}&\lambda&0_{1,1}\\
0_{1,m}&0_{1,1}&1\\
\end{bmatrix}
\end{align*}
where the matrix $ Y_{m,m} $ is given in \eqref{ext}.
\begin{theorem}
 With the above notation,
the linear mapping
\begin{center}
\begin{tabular}{lll}
$\rho\colon$  & $\mathcal{H}^m_\lambda$ &$\longrightarrow  End(V)$\\
           & $x$&$\longmapsto [x,\cdot]_F$
\end{tabular}
\end{center}
\end{theorem}
is a faithful representation of $ \mathcal{H}^m_\lambda. $ We will call the  minimal faithful representation for Heisenberg Hom-Lie algebra $ \mathcal{H}^m_\lambda. $

\subsection{Cohomology of Heisenberg  Hom-Lie algebra with respect to   minimal faithful
 representation
   }
A straightforward calculation shows that
\[ \delta^1_r\circ \delta^0_r(v) =0\qquad \forall 
v\in V\] Then, we can define the space $ H^1(\mathcal{H}^m_\lambda,F) =Z^1(\mathcal{H}^m_\lambda,F)/B^1(\mathcal{H}^m_\lambda,F).$
\begin{thm}
	Let $T:\mathcal{H}^m_{\lambda}\to F $ be linear. Then $ T $ is a no trivial  $1$-cocycle  if and only if the matrix representation of $T$ in the 
	ordered bases $ (x_1,\cdots,x_m, y_1,\cdots,y_m,z)$ (given in \eqref{ext} ) and $ \left(\beta^r(v_1),\dots,\beta^r(v_{m+2}) \right) $ is of the form 
	\[\begin{bmatrix}
	a_{11}&a_{12}&\dots&a_{1m}&0&.&\dots&0&0\\
	a_{21}&a_{22}&\dots&a_{2m}&\vdots&\vdots&\hdots&\vdots&0\\
	\vdots&\vdots&\dots & \vdots & 0 &	\dots&\dots&0 &0\\
	a_{m1}&a_{m2}&\dots&a_{mm}&b_{m+1,1}&.&\dots&b_{m+1,m}&0\\
	0&0&\dots&0&0&\dots&.&0&0\\
	0&0&\dots&0	&0&\dots&0&0&0
	\end{bmatrix}\]
	With $a_{ij}=a_{ji}$ for all $i,j\in \{1,\cdots,m\}$.
\end{thm}
\begin{cor}
	\[	\dim H^1\left( \mathcal{H}^m_{\lambda},F\right) =\frac{m(m+3)}{2}.\]
\end{cor}

Now we give the spaces $B^1_{r,\beta}\left( \mathcal{H}^m_{\lambda},F\right)  $, $Z^1_{r,\beta}\left( \mathcal{H}^m_{\lambda},F\right)  $  and $H^1_{r,\beta}\left( \mathcal{H}^m_{\lambda},F\right)$ and their dimension.
 Let $ \mu\in \C $  and  $ E(\mu) =\{u\in E/ Xu=\mu u\}$.\\
Let $ (u_1,\cdots,u_{n},v_{m+2}) $ is a basis of the space of $ 0 $-hom-cochain \[ C_{\alpha,\beta}^{0}(\mathcal{H}^m_\lambda,F) =\{v\in F/\,\beta(v)=v\}\] such that $ u_k=\displaystyle \sum_{i=1}^{m+1}\mu_{ik}v_i $.
Let $f_{k},\ g:\mathcal{H}^m_{\lambda}\to F$  be the
linear transformations respectively defined by
	\begin{align*}
	f_{k}(x_p)&=\mu_{pk}\lambda^{r}v_{m+1},\quad &f_{k}(y_p)&=0,& f_{k}(z)&=0, \\
		g(x_p)&=0,\quad &g(y_p)&=\beta^{r}(v_{p})& g(z)&=\lambda^{r}v_{m+1}.\\
	\end{align*}
\begin{thm}
	With above notations, associated to the  minimal faithful representation of the Heisenberg  Hom-Lie algebra $ \mathcal{H}^m_\lambda $, we have
	\[ B^1_{r,\beta}\left( \mathcal{H}^m_{\lambda},F\right)
	 =\bigoplus_{1\leq k\leq n}<f_{k}>\bigoplus<g> .\] 
\end{thm}
\begin{cor}
	\[ \dim\left(  B^1_{r,\beta}( \mathcal{H}^m_{\lambda},F)     \right)=\dim E(1)+1  \]
\end{cor}
\begin{theorem}
Let $ f:\mathcal{H}^m_\lambda\to F  $	
be defined by the matrix representation \[ \begin{bmatrix}
A_m&B_{m}&C_{m,1}\\
U_{1,m}&V_{1m}&D_{m,1}\\
W_{1,m}&T_{1,m}&\mu\\
\end{bmatrix} .\]
Then, $ f $	 is a $1$-Hom-cochain  if and only if the following conditions are satisfied :
\begin{align*}
\fourIdx{t}{}{}{}{X}\,A_m X&=\lambda A_m, \quad& \fourIdx{t}{}{}{}{X}\, \fourIdx{t}{}{}{}{U_{1m}}&=\lambda\, \fourIdx{t}{}{}{}{U_{1m}},\quad&
\fourIdx{t}{}{}{}{X}\, \fourIdx{t}{}{}{}{W_{1m}}\,   &=\fourIdx{t}{}{}{}{W_{1m}}, \quad&
\fourIdx{t}{}{}{}{X}\, B_{m}&=B_{m}\,\fourIdx{t}{}{}{}{X},\\ 
X\fourIdx{t}{}{}{}{V_{1m}}&=\fourIdx{t}{}{}{}{V_{1m}},\quad& X\fourIdx{t}{}{}{}{T_{1m}}&=\fourIdx{t}{}{}{}{T_{1m}},\quad&
\fourIdx{t}{}{}{}{X}\,C_{m1}&=C_{m1}&\quad \lambda\mu&=\mu. 
\end{align*}
\end{theorem}


\begin{thm}
	Let $T:\mathcal{H}^m_{\lambda}\to F $ be linear. Then $ T $ is a $1$-Hom-cocycle  if and only if  the matrix representation of $T$ in the 
	ordered bases $ (x_1,\cdots,x_m, y_1,\cdots,y_m,z)$ (given in \eqref{ext} ) and $ \left(\beta^r(v_1),\dots,\beta^r(v_{m+2}) \right) $ is of the form 
	\[\begin{bmatrix}
	A_{m}&b I_{m}&0_m\\
	U_{1m}&V_{1m}&b\\
	0_{1m}&0_{1m}&0
	\end{bmatrix}\]
	with $ \fourIdx{t}{}{}{}{A_m}\,=A_m$,$\quad   \fourIdx{t}{}{}{}{X}\,A_m\, X =\lambda A_m$, $\quad	 \fourIdx{t}{}{}{}{X}\, \fourIdx{t}{}{}{}{U_{1m}}\,=\lambda \fourIdx{t}{}{}{}{U_{1m}}  $ and $X\,\fourIdx{t}{}{}{}{V_{1m}}=  \fourIdx{t}{}{}{}{V_{1m}} $.
\end{thm}
\begin{thm}
	Let $T:\mathcal{H}^m_{\lambda}\to F $ be linear. Then $ T $ is a no trivial  $1$-Hom--cocycle  if and only if the matrix representation of $T$ in the 
	ordered bases $ (x_1,\cdots,x_m, y_1,\cdots,y_m,z)$ (given in \eqref{ext} ) and $ \left(\beta^r(v_1),\dots,\beta^r(v_{m+2}) \right) $ is of the form 	
		\[\begin{bmatrix}
	A_{m}&0_{m}&0\\
	0_{1m}&V_{1m}&0\\
	0_{1m}&0_{1m}&0
	\end{bmatrix}\]
	with $ \fourIdx{t}{}{}{}{X}\,A_m\, X =\lambda A_m$ 	 and $X\,\fourIdx{t}{}{}{}{V_{1m}}=  \fourIdx{t}{}{}{}{V_{1m}} $
	
\end{thm}

\begin{thm}	
	Let $ I=\{\lambda_i;\,\chi(\lambda_i)= \chi(\frac{\lambda}{\lambda_i})=0 \}=\{\lambda_1,\cdots,\lambda_p,\frac{\lambda}{\lambda_1}
	,\cdots,\frac{\lambda}{\lambda_p}    \}$  ($p\in\N$ and $card(I)=2p $) . Then, 
	\[ \dim(H^1_{r,\beta}\left( \mathcal{H}^m_{\lambda},F)\right)=\dim(E(1))+\displaystyle \sum_{i=1}^{p}
	dim(E(\lambda_i))\dim(E(\frac{\lambda}{\lambda_i})).\]
\end{thm}
\section{The trivial representation of Heisenberg  Hom-Lie algebras}
Let $ (v_1,\dots,v_{n}) $ be  a basis for a vector space $ T $.
We assume that the representation $(T,[\cdot,\cdot]_{T},\beta)$ of the 
Heisenberg  Hom-Lie algebra $\mathcal{H}^m_\lambda$
 is trivial.
Since $ [\cdot,\cdot] _T = $ 0, the operator defined in \eqref{def cobbb} becomes
\begin{align*}
& \delta^{k}_T(f)(x_{0},\dots,x_{k}) =& \sum_{0\leq s < t\leq k}(-1)^{t} 
f\Big(\alpha(x_{0}),\dots,\alpha(x_{s-1}),[x_{s},x_{t}],\alpha(x_{s+1}),\dots,\widehat{x_{t}},\dots,\alpha(x_{k})\Big). \nonumber
\end{align*}
A straightforward calculation shows that \[ \delta^2\circ\delta^1(f)=0,\qquad \forall f\in C^1(\mathcal{H}^m_\lambda,T) \]
\begin{prop}
	\[ B^2(\mathcal{H}^m_\lambda,T) =\bigoplus_{1\leq k\leq n}<f_{k}>,\ where\]
	$ f_k(x_{i},y_j)=\delta_{i,j}v_k,\quad f_k(x_i,x_j)=f_k(y_i,y_j)=f_k(z,x_i)=f_k(z,y_i)=0$\ 
\end{prop}
\begin{prop}
	Let $ f:\mathcal{H}^m_\lambda\times\mathcal{H}^m_\lambda \to T $ be a  bilinear, skew-symmetric map. Then $f$ is a $ 2 $-cocycle on $T$ if and only if  $ f(z,u) =0\quad \forall u\in \mathcal{H}^m_\lambda$.
\end{prop}
\begin{cor}
\[ \dim Z^2(\mathcal{H}^m_\lambda,T)=(2m^2-m)\dim(T) \]
\end{cor}
\begin{thm}
		Let $ f:\mathcal{H}^m_\lambda\times\mathcal{H}^m_\lambda \to T $ be a  bilinear, skew-symmetric map. Then $f$ is a no trivial $ 2 $-cocycle on $T$ if and only if  $ f(z,u) =0\quad \forall u\in \mathcal{H}^m_\lambda$ and it satisfies one of the following conditions.
		\begin{enumerate}[(i)]
						\item  There exists $ i,\, j\in \{1,\cdots,m\} $ such that $f(x_i,y_i) \neq f(x_j,y_j) $.
						\item $ f(E,E)\neq 0. $
							\item $ f(E^*,E^*)\neq 0. $
		\end{enumerate}
\end{thm}

\begin{cor}
\[ \dim H^2(\mathcal{H}^m_\lambda,T)=(2m^2-m-1)\dim(T) .\]
\end{cor}

\section{ The Adjoint representation of Heisenberg  Hom-Lie algebras} 
Associated to the $ \alpha^r $-hom-adjoint representation, the operator  $\delta^{k}:C^{k}(\mathcal{H}^m_\lambda ,\ \mathcal{H}^m_\lambda )\rightarrow C^{k+1}(\mathcal{H}^m_\lambda \ \mathcal{H}^m_\lambda )$  is given by 
\begin{align*}
&\delta^{k}_r(f)(x_{0},\dots,x_{k})=\nonumber \\&\sum_{0\leq s < t\leq k}(-1)^{t}
f\Big(\alpha(x_{0}),\dots,\alpha(x_{s-1}),[x_{s},x_{t}],\alpha(x_{s+1}),\dots,\widehat{x_{t}},\dots,\alpha(x_{k})\Big) \nonumber \\
&+\sum_{s=0}^{k}(-1)^{s}\Bigg[\alpha^{k+r-1}(x_{s}), f\Big(x_{0},\dots,\widehat{x_{s}},\dots,x_{k}\Big)\Bigg],\label{def cobbb}
\end{align*}
by straightforward computations, we have
\[\delta^2_r\circ\delta^1_r(f) =0\qquad \forall f\in C^m(\mathcal{H}^m_\lambda,\mathcal{H}^m_\lambda) \]
Associated to the representation $ad_r$, we obtain the complex $ (C^k(\mathcal{H}^m_\lambda,\mathcal{H}^m_\lambda),\delta^{k}_r)_{0\leq k\leq 2} $. denote the set of closed $k$-cochains by 
$ Z^k(\mathcal{H}^m_\lambda,ad_r) $ and the set of exact $k$-cochains by $ B_k(\mathcal{H}^m_\lambda,ad_r) $. Denote the corresponding  cohomology by \[H^ k(\mathcal{H}^m_\lambda,ad_r)=Z^k(\mathcal{H}^m_\lambda,ad_r)/B^k(\mathcal{H}^m_\lambda,ad_r) .\]
\subsection{Second cohomology group of  $ \mathcal{H}^1_\lambda$ }
Any no trivial $2$-cocycle with The adjoint representation of $ \mathcal{H}^1_\lambda $ is given by bilinear map
 $f: \mathcal{H}^1_\lambda\times\mathcal{H}^1_\lambda\to  \mathcal{H}^1_\lambda$ defined, with respect to the  basis $ (\alpha^{r+1}(x),\alpha^{r+1}(y),\alpha^{r+1}(z)) $, by  
\begin{align*}
f(x,y)&=0\\
f(x,z)&=a\,\alpha^{r+1}(x)+b\,\alpha^{r+1}(y)+c\, z\\
f(y,z)&=d\,\alpha^{r+1}(x)-a\,\alpha^{r+1}(y)+e\, z
\end{align*} 
Where $ a,b,c, d, e $ are parameters in $ \C. $
Therefore $ \dim H^2(\mathcal{H}^1_{\lambda},ad_r) =5.$\\

We denote  $ m(\nu) $ 
the
multiplicity of  a root $ \nu $
of the
characteristic
polynomial $ \chi $ of $ \alpha $ ( if $ \chi(\nu) \neq 0 $ we write $ m(\nu)=0 $).
The following tables gives the possibilities for $ \dim H^2_{h}(\mathcal{H}^1_{\lambda},ad_r)   $\\

 \begin{tabular}{|c|c|}
 	\hline
 	Matrix $ S $ & Dimension \\
 	
 	\hline
 	&\\

 	$\begin{pmatrix}
 	1&0&0\\
 	0&\lambda&0\\
 	0&0&\lambda
 	\end{pmatrix},$
 	where  $ \lambda\notin\{-1,1\}$& $\dim H^2_{h}(\mathcal{H}^1_{\lambda},ad_r) =2$
 
 	\\
 	& \\
 	\hline
 	&\\
 	

 $\begin{pmatrix}
 \lambda&0&0\\
 0&1&0\\
 0&0&\lambda
 \end{pmatrix},$
 where  $ \lambda\notin\{-1,1\}$& $\dim H^2_{h}(\mathcal{H}^1_{\lambda},ad_r) =2$
 
 \\
 & \\
 \hline
 &\\
 
 	
 	$\begin{pmatrix}
 	a&0&0\\
 	0&\frac{\lambda}{a}&0\\
 	0&0&\lambda
 	\end{pmatrix},$
 	where  $ a\in\C^*$& $\dim H^2_{h}(\mathcal{H}^1_{\lambda},ad_r) =m(1)+m(-1)$
 	
 	\\
 	& \\
 	\hline
 	&\\
 	
 	$ \begin{pmatrix}
 	1&1&0\\
 	0&1&0\\
 	0&0&1
 	\end{pmatrix},$

 	&$\dim H^2_{h}(\mathcal{H}^1_{\lambda},ad_r)= 2$  \\
 	&\\
 	
 	\hline
 	&\\
 	
 	$ \begin{pmatrix}
 	a&1&0\\
 	0&a&0\\
 	0&0&\lambda
 	\end{pmatrix},$ where  $ a^2=\lambda $ and $ a\neq 1 $
 	&$\dim H^2_{h}(\mathcal{H}^1_{\lambda},ad_r) = 1$  \\
 	&\\
 	
 	\hline
 \end{tabular}

 \subsection{Second cohomology group of  $ \mathcal{H}^2_\lambda$ }
 Any no trivial $2$-cocycle with The adjoint representation of $ \mathcal{H}^2_\lambda $ is given by bilinear map
 $f\colon \mathcal{H}^2_\lambda\times\mathcal{H}^1_\lambda\to  \mathcal{H}^2_\lambda$ defined, with respect to the  basis $ (\alpha^{r+1}(x_1),\alpha^{r+1}(x_2),\alpha^{r+1}(y_1),\alpha^{r+1}(y_2),\alpha^{r+1}(z)) $, by  
 \begin{gather*}
 \varphi(x_1,x_2)
 = a_{1,(1,2)}\alpha^{r+1}(x_1)
 +a_{2,(1,2)}\alpha^{r+1}(x_2)
 + b_{1,(1,2)}\alpha^{r+1}(y_1)
 + b_{2,(1,2)}\alpha^{r+1}(y_2)\\
 \varphi(x_1,y_2)
 = c_{1,(1,2)}\alpha^{r+1}(x_1)
 +c_{2,(1,2)}\alpha^{r+1}(x_2)
 + d_{1,(1,2)}\alpha^{r+1}(y_1)
 + d_{2,(1,2)}\alpha^{r+1}(y_2)\\
 \varphi(x_2,y_1)= c_{1,(2,1)}\alpha^{r+1}(x_1)
 +c_{2,(2,1)}\alpha^{r+1}(x_2)
 + d_{1,(2,1)}\alpha^{r+1}(y_1)
 + d_{2,(2,1)}\alpha^{r+1}(y_2)\\
 \varphi(y_1,y_2)= e_{1,(1,2)}\alpha^{r+1}(x_1)
 +e_{2,(1,2)}\alpha^{r+1}(x_2)
 + f_{1,(1,2)}\alpha^{r+1}(y_1)
 +f_{2,(1,2)}\alpha^{r+1}(y_2)\\
 \varphi(x_1,y_1)=-\left( f_{2,(1,2)}
 +c_{1,(2,2)}-c_{2,(2,1)}-c_{1,(1,1)}
 \right)
 \alpha^{r+1}(x_1)
 +\left( f_{1,(1,2)}
 +c_{1,(1,2)}
 \right)	\alpha^{r+1}(x_2)\\
 +\left(- d_{1,(2,2)}
 +a_{2,(1,2)}+d_{2,(1,2)}+ d_{1,(1,1)}
 \right)	\alpha^{r+1}(y_1)
 +\left(d_{1,(2,1)}
 -a_{1,(1,2)}
 \right) \alpha^{r+1}(y_2)\\
 \varphi(x_2,y_2)
 =\left(c_{2,(2,1)}- f_{2,(1,2)}
  \right)\alpha^{r+1}(x_1)
 -\left(
 -c_{2,(2,2)}- f_{1,(1,2)}
 -c_{1,(1,2)}+c_{2,(1,1)}\right)
 \alpha^{r+1}(x_2)\\
 +\left(
 a_{2,(1,2)}+d_{2,(1,2)}\right)
 \alpha^{r+1}(y_1)
 +\left( d_{1,(2,1)}
 -a_{1,(1,2)}-d_{2,(1,1)}+ d_{2,(2,2)}\right)\alpha^{r+1}(y_2).
 \end{gather*}
 Therefore $ \dim H^2(\mathcal{H}^2_{\lambda},ad_r) =20.$\\
Assume that $ X $ has two complex eigenvalues $ \mu, $ $ \mu' $; Then $S$ takes the form 
\begin{multicols}{2}
	\begin{enumerate}[(i)]
		\item \label{hom}	$ \begin{pmatrix}
		\mu&0&0&0\\
		0&\mu'&0&0\\
		0&0&\frac{\lambda}{\mu}&0\\
		0&0&0&\frac{\lambda}{\mu'}
		\end{pmatrix}$  \qquad  \qquad  or 
		\item  	$ \begin{pmatrix}
		\mu&-\frac{\mu^2}{\lambda}&0&0\\
		0&\mu&0&0\\
		0&0&\frac{\lambda}{\mu}&0\\
		0&0&1&\frac{\lambda}{\mu}
		\end{pmatrix}$ $ (\mu'=\mu). $\label{hom2}
	\end{enumerate}
\end{multicols}
The following table gives the possibilities for 
$ \dim H^2_{h}(\mathcal{H}^2_{\lambda},ad_r) $ according  to the matrix of case \eqref{hom}

\begin{tabular}{|c|c|}
	\hline
	Spectrum $ \sigma(X) $  & Dimension \\
	\hline
	&\\
	$ \sigma(X)=\{-1,1\} ;$	
	$ \lambda^2=1$&\\ 
	\cline{1-1}	
	$ \sigma(X)=\{-1\} ;$		
	$ \lambda=-1$& $\dim H^2_{h}(\mathcal{H}^1_{\lambda},ad_r) =10$
	\\ 
	\cline{1-1}	
	$ \sigma(X)=\{1\} ;$		
	$ \lambda=-1$&	
	\\
	& \\
	\hline
	&\\
	$\sigma(X)=\{\lambda,1\} ;  $
	$ \lambda^2\neq 1$&\\ 
	\cline{1-1}	
	$ \sigma(X)=\{\lambda\} ;$		
	$ \lambda^2\neq1$& $\dim H^2_{h}(\mathcal{H}^1_{\lambda},ad_r) =8$\\ 
	\cline{1-1}	
	$ \sigma(X)=\{1\} ;$		
	$ \lambda^2\neq1$&
	\\
	& \\
	\hline
	&	\\
	$\sigma(X)=\{\lambda^2,1\}; $
	$ \lambda\in\{j,j^2\}$&\multirow{1}*{$\dim H^2_{h}(\mathcal{H}^1_{\lambda},ad_r)= 7$}\\ 
	\cline{1-1}
	$\sigma(X)=\{\lambda,\lambda^2\};$ $\lambda\in\{j,j^2\}$ & 
	\\
	\hline
	&	\\
	$\sigma(X)=\{\mu,1\}; $ $ \mu^2\neq 1; $
	$ \lambda=1$ 
	& $\dim H^2_{h}(\mathcal{H}^1_{\lambda},ad_r)= 6$\\
	&
	\\
	\hline
\end{tabular}

\begin{tabular}{|c|c|}
	\hline	
	$\sigma(X)=\{1,\lambda^3\} ;  $
	$ \lambda\in\{-i,i\}$&\\ 
	\cline{1-1}	
	$ \sigma(X)=\{\lambda,\lambda^3\} ;$	$ \lambda\in\{-i,i\}$	
	& \\ 
	\cline{1-1}	
	$ \sigma(X)=\{-1,1\} ;$		
	$ \lambda\in\{-i,i\}$&
	\\
	\cline{1-1}	
	$ \sigma(X)=\{-1,\lambda\} ;$		
	$ \lambda\in\{-i,i\}$&$\dim H^2_{h}(\mathcal{H}^1_{\lambda},ad_r) =5$
	\\
	\cline{1-1}	
	$ \sigma(X)=\{1,\sqrt{r}e^{i\frac{\theta}{2}}\} ;$ $ \lambda=re^{i\theta} $; $
	\theta\in\intervallefo{0}{2\pi} $;		
	$ \lambda\notin\{-1,1,j^2\}$&\\
	\cline{1-1}	
	$ \sigma(X)=\{1,-\sqrt{r}e^{i\frac{\theta}{2}}\} ;$ $ \lambda=re^{i\theta} $; $ \theta\in\intervalleff{0}{2\pi} $		
	$ \lambda\notin\{-1,1,j\}$&\\
	\cline{1-1}	
	$ \sigma(X)=\{re^{i\theta},\sqrt{r}e^{i\frac{\theta}{2}}\} ;$ $ \lambda=re^{i\theta} $; $ \theta\in\intervalleff{0}{2\pi} $		
	$ \lambda\notin\{-1,1,j^2\}$&\\
	\cline{1-1}	
	$ \sigma(X)=\{re^{i\theta},-\sqrt{r}e^{i\frac{\theta}{2}}\} ;$ $ \lambda=re^{i\theta} $; $ \theta\in\intervalleff{0}{2\pi} $		
	$ \lambda\notin\{-1,1,j\}$&\\
	&\\
	\hline
	&\\
	$\sigma(X)=\{1,\frac{1}{\lambda}\} ;  $
	$ o(\lambda)>4$&\\ 
	\cline{1-1}	
	$ \sigma(X)=\{-1,1\} ;$	$ o(\lambda)>4$	
	& \\ 
	\cline{1-1}	
	$ \sigma(X)=\{1,\lambda^2\} ;$		
	$ o(\lambda)>4$&
	\\
	\cline{1-1}	
	$ \sigma(X)=\{-1,\lambda\} ;$		
	$ o(\lambda)>4$&
	\\
	\cline{1-1}	
	$ \sigma(X)=\{1,-\lambda \};$ 	
	$ o(\lambda)>4$&$\dim H^2_{h}(\mathcal{H}^1_{\lambda},ad_r) =4$\\
	\cline{1-1}	
	$ \sigma(X)=\{\lambda,\frac{1}{\lambda}\} ;$ $o (\lambda)>4 $;&\\
	\cline{1-1}	
	$ \sigma(X)=\{\lambda,-\lambda\} ;$		
	$o (\lambda)>4$&\\
	\cline{1-1}	
	$ \sigma(X)=\{\lambda,\lambda^2\} ;$		
	$ o(\lambda)>4$&\\	
	\cline{1-1}	
	$ \sigma(X)=\{\mu,\mu^2\} ;$		
	$ o(\mu)\in\{3,4,5\}$; $ \lambda=1 $&\\		
	\cline{1-1}	
	$ \sigma(X)=\{\mu,\mu^3\} ;$		
	$ o(\mu)=5$; $ \lambda=1 $&\\	
	&\\
	\hline
	\hline
$\sigma(X)=\{\mu,-1\} ;  $ $ \mu^2=-\lambda; $ $ \lambda^3=-1; $
$ o(\mu)\in\{3,6\}$&\\ 
\cline{1-1}	
$ \sigma(X)=\{\mu,1\} ;$	$ \lambda\notin\{1,\mu^2\}$; $ \mu\notin\{-1,1,\lambda,\frac{1}{\lambda},\lambda^2\} $	
& \\ 
\cline{1-1}	
$ \sigma(X)=\{\mu,\lambda\} ;$		$ \lambda\notin\{1,\mu,\mu^2,-\mu,\frac{1}{\mu}\}$; $ \mu\notin\{-1,1,\lambda,\lambda^2\} $		
&
\\
\cline{1-1}	
$ \sigma(X)=\{\lambda^5,\lambda^3\} ;$		
$ o(\lambda)=6$&
\\
\cline{1-1}	
$ \sigma(X)=\{\mu,-\lambda \};$ $ \lambda=-\mu^2 $	
$ o(\mu)=6$&\\
\cline{1-1}	
$ \sigma(X)=\{\mu,-\mu^2\} ;$ $ \lambda=-\mu$;  $ o(\mu)=3 $&\\
\cline{1-1}	
$ \sigma(X)=\{j,j^2\} ;$		
$\lambda=1$&$\dim H^2_{h}(\mathcal{H}^1_{\lambda},ad_r) =3$\\
\cline{1-1}	
$ \sigma(X)=\{-\lambda,\lambda^2\} ;$		
$ o(\lambda)=6$&\\	
\cline{1-1}	
$ \sigma(X)=\{\lambda^3,\lambda^5\} ;$		
$ o(\lambda)=6$	&\\
\cline{1-1}	
$ \sigma(X)=\{\lambda^2,\lambda^3\} ;$		
$ o(\lambda)=6$	
&\\
\cline{1-1}		
$\sigma(X)=\{\mu\} ;  $ $ \mu^3=\lambda; $ $ \lambda\neq1; $&\\ 
\cline{1-1}
$\sigma(X)=\{\mu\} ;  $ $ \mu^3=\lambda^2; $ $ \lambda\neq1; $&\\ 
&\\
\hline
\end{tabular}
\begin{tabular}{|c|c|}
\hline	
$\sigma(X)=\{\mu,\mu^3\} ;  $ $ \mu^5=\lambda $, $ o(\mu)>6 $	
&\\ 
\cline{1-1}	
$\sigma(X)=\{\mu,\mu^2\} ;  $ $ \mu^4=\lambda $, $ o(\mu)>4 $	
&\\ 
\cline{1-1}		
$\sigma(X)=\{\mu,\frac{\lambda}{\mu^2}\} ;  $ $ \mu^5=\lambda^2 $, $ \lambda\notin \{1,\mu,\mu^2\} $,   $ o(\mu)>3 $	
&\\ 
\cline{1-1}		
$\sigma(X)=\{\mu,\frac{1}{\mu^2}\} ;  $
$ \lambda=1 $,  $ o(\mu)>5 $	
&\\ 
\cline{1-1}			
$\sigma(X)=\{\mu,\lambda\mu\} ;  $
$ \lambda^3\neq 1 $,  $ o(\mu)=3 $	
&\\ 
\cline{1-1}		
$\sigma(X)=\{\mu,\frac{\mu^2}{\lambda}\} ;  $
$ \lambda\notin\{-\mu,1,\mu\} $,  $ o(\mu)>2$	
&\\ 
\cline{1-1}	
$\sigma(X)=\{\mu,\frac{\mu^2}{\lambda}\} ;  $
$ o(\mu)>3 $,  $ o(\frac{\mu}{\lambda})=3$	
&\\ 
\cline{1-1}
$\sigma(X)=\{\mu,\mu^2\} ;  $
$ o(\mu)>5 $,  $ \lambda=1$	
& $\dim H^2_{h}(\mathcal{H}^1_{\lambda},ad_r) =2$\\ 
\cline{1-1}	
$\sigma(X)=\{\mu,\frac{\mu^2}{\lambda}\} ;  $
$ \mu^4=\lambda^3 $,  $ \lambda\notin\{1,\mu^2\}$, $ \mu\neq 1 $	
&\\ 
\cline{1-1}	
$\sigma(X)=\{\mu,\frac{\mu^2}{\lambda}\} ;  $
$ \mu^3=\lambda^2 $,  $ \lambda\notin\{1,-\mu,\mu,\mu^2\}$, $ \mu\neq 1 $	
&\\ 
\cline{1-1}	
$\sigma(X)=\{\mu,\frac{\mu^2}{\lambda}\} ;  $
$ \mu^5=\lambda^4 $,  $ \lambda\notin\{1,\mu,\mu^2\}$, $ \mu^3\neq1$	
&\\ 
\cline{1-1}	
$\sigma(X)=\{-1,\mu\} ;  $
$ \mu^2=-\lambda $,  $ \lambda^3\neq-1$, 
&\\ 
\cline{1-1}	
$\sigma(X)=\{\mu,\mu^2\} ;  $
$ \mu^4=\lambda $,  $ o(\mu)>4$, 
&\\ 
\cline{1-1}	
$\sigma(X)=\{\mu,\mu^2\} ;  $
$ \mu^5=\lambda $,  $ o(\mu)>6$, 
&\\ 
\cline{1-1}	
$\sigma(X)=\{\mu,\mu^2\} ;  $
$ \lambda=1 $,  $ o(\mu)>5$, 
&\\ 
\cline{1-1}	
$\sigma(X)=\{\mu,\mu^2\} ;  $
,  $ o(\mu)=3$, $ \lambda\notin\{-\mu,1,\mu,\mu^2\} $, $ \lambda^2\neq \mu $
&\\
\cline{1-1}	
$\sigma(X)=\{\mu,\mu^2\} ;  $
,  $ -\mu^2=\lambda$, $ \lambda\neq1$, $ \lambda^3\neq-1 $
&\\
\cline{1-1}	
$\sigma(X)=\{\mu,\mu^2\} ;  $
,  $ \mu^5=\lambda^2$, $ \lambda\notin\{1,\mu\}$, $ o(\mu) >3$
&\\
\cline{1-1}	
$\sigma(X)=\{\mu\} ;  $
,  $ \mu^3=\lambda$, $ o(\mu) >3$
&\\
\cline{1-1}	
$\sigma(X)=\{\mu\} ;  $
,  $ \mu^3=\lambda^2$, $ \mu\neq1$
&\\

&\\
\hline
$\sigma(X)=\{\mu,\frac{\lambda}{\mu^2}\} ;  $
,  $ \lambda\notin\{-\mu^2,1,\mu,\mu^2,\mu^3,\mu^4,\mu^5\}$, $ o(\mu) >3$
&\\
\cline{1-1}	
$\sigma(X)=\{\mu,\frac{\mu^2}{\lambda}\} ;  $
$ \lambda\notin\{1,\mu^2,-\mu,\mu,j\mu,j^2\mu,\}$,&\\ $ \mu\neq1$, $ \mu^3\neq\lambda^2$, $ \mu^4\neq\lambda^3$, $ \mu^5\neq4$, , $ \mu^4\neq\lambda^2$, $ \mu^5\neq\lambda^3$&$\dim H^2_{h}(\mathcal{H}^1_{\lambda},ad_r) =1$\\
\cline{1-1}	
$\sigma(X)=\{\mu,\mu^2\} ;  $
$ \lambda\notin\{1,\mu,\mu^2,-\mu^2,\mu^3,\mu^4,\mu^5\}$,&\\ $ o(\mu)>3$, $ \mu^5\neq\lambda^2$,
&\\
\hline
in other cases&$\dim H^2_{h}(\mathcal{H}^1_{\lambda},ad_r) =0$\\
\hline				
\end{tabular}
 The following table gives the possibilities for 
 $ \dim H^2_{h}(\mathcal{H}^2_{\lambda},ad_r) $ according  to the matrix of case \eqref{hom2}\\

 \begin{tabular}{|c|c|}
 		\hline
 	Spectrum $ \sigma(X) $  & Dimension \\
 	\hline
 	\hline	
 	$\sigma(X)=\{1\} ;  $ $ \lambda=1 $,&  $\dim H^2_{h}(\mathcal{H}^1_{\lambda},ad_r) =8$\\
\hline 	
$\sigma(X)=\{-1\} ;  $ $ \lambda=-1 $,&  $\dim H^2_{h}(\mathcal{H}^1_{\lambda},ad_r) =4$\\
\hline  	
$\sigma(X)=\{1\} ;  $ $ \lambda\neq1 $,&  $\dim H^2_{h}(\mathcal{H}^1_{\lambda},ad_r) =3$\\
	\cline{1-1}
$\sigma(X)=\{\lambda\} ;  $ $ \mu^2\neq1 $,&\\	
\hline  	
$\sigma(X)=\{\mu\} ;  $ $ \lambda=1 $, $ o(\mu)=3 $&  $\dim H^2_{h}(\mathcal{H}^1_{\lambda},ad_r) =2$\\
\hline 
$\sigma(X)=\{-\lambda\} ;  $ $ \lambda^2=-1 $, $ o(\mu)=3 $&\\
\cline{1-1}
$\sigma(X)=\{\mu\} ;  $ $ \mu^3=\lambda $, $ \lambda\neq1 $& $\dim H^2_{h}(\mathcal{H}^1_{\lambda},ad_r) =1$\\
\cline{1-1}
$\sigma(X)=\{\mu\} ;  $ $ \mu^3=\lambda^2 $, $ \lambda\neq1 $, $ \mu^2\neq1 $&\\
 	\hline				
 \end{tabular}
  \subsection{Second cohomology group of
  	  $ \mathcal{H}^m_\lambda\quad (m>2)$ }
Let $f_{i,j},\ g_{i,j},\  h_{i,j},\ l_{i},\ t_{i},\  m_{i}   $
 be a skew-symmetric bilinear functions  defined respectively by \\
\begin{tabular}{ccc}
$f_{i,j}(x_j,y_i)=\alpha^r(z);$&
$g_{i,j}(x_i,x_j)=\alpha^r(z);$&
$h_{i,j}(y_i,y_j)=\alpha^r(z);$\\
$l_{i}(x_p,y_p)=-\alpha^r(x_i);$& $ l_i(y_i,z)=-\alpha^r(z) $;&\\
$t_{i}(x_p,y_p)=-\alpha^r(y_i) ;$& $t_{i}(x_i,z)=\alpha^r(z) ;$& \\
$m(x_p,y_p)=-\alpha^r(z)$.&  
\end{tabular}

The image of the other basic elements is zero or deduced by skew symmetric.
\begin{lem}
With above notations, associated to the adjoint representation of the Heisenberg  Hom-Lie algebra $ \mathcal{H}^m_\lambda $, we have 
\begin{align*}
B^2(\mathcal{H}^m_\lambda,\mathcal{H}^m_\lambda)&=
\bigoplus_{1\leq i,j\leq m}<f_{i,j}>\bigoplus_{1\leq i<j\leq m}<g_{i,j}>\bigoplus_{1\leq i<j\leq m}<h_{i,j}>\\
&\bigoplus_{1\leq i\leq m}<l_{i}>\bigoplus_{1\leq i\leq m}<t_{i}>\bigoplus<m>
\end{align*}
\end{lem}

\begin{thm}
	The space $ 	H^2(\mathcal{H}^m_\lambda,\mathcal{H}^m_\lambda) $ is generated by \[ \{f_{j,(i,l)},g_{i,(l,j)},h_{j,(i,l)},k_{i,(j,l)},  u_{i,(j,l)}, a_{l,(i,j)},  b_{i,(j,l)},  c_{j,(i,l)}, f_{j},h_j, k_{j}, a_{j}
	  \} \]
	where,\\	
\begin{tabular}{rrrr}
 	$f_{j,(i,l)}(x_i,x_j)=\alpha^{r+1}(y_l),$& 	
	&$f_{j,(i,l)}(x_i,x_l)=\alpha^{r+1}(y_j)$& $(i\neq j,\ i\neq l,\ j\neq l)  $;\\
	 	$g_{i,(l,j)}(x_i,x_j)=\alpha^{r+1}(y_l),$& 
		&$g_{i,(l,j)}(x_l,x_j)=\alpha^{r+1}(y_i)$&$(i\neq j,\ i\neq l,\ j\neq l)  $;\\
	$	h_{j,(i,l)}(x_i,x_j)=\alpha^{r+1}(x_l)$,& 	
	&	$	h_{j,(i,l)}(x_j,y_l)=\alpha^{r+1}(y_i)$&$(i\neq j,\ i\neq l,\ j\neq l)  $;\\	
	$	k_{i,(j,l)}(x_i,y_l)=\alpha^{r+1}(y_j)$,& 	
	&	$	k_{i,(j,l)}(x_j,y_l)=\alpha^{r+1}(y_i)$&$(i\neq j,\ i\neq l,\ j\neq l)  $;\\	
	$	u_{i,(j,l)}(y_j,y_l)=\alpha^{r+1}(y_i)$&
	&	$	u_{i,(j,l)}(x_i,y_j)=	\alpha^{r+1}(x_l)$,&$(i\neq j,\ i\neq l,\ j\neq l)  $;\\
$a_{l,(i,j)}(y_i,y_l)=\alpha^{r+1}(x_j)$,& 
&$	a_{l,(i,j)}(y_i,y_j)=\alpha^{r+1}(x_l)$&$(i\neq j,\ i\neq l,\ j\neq l)  $;\\
$b_{i,(j,l)}(y_j,y_l)=\alpha^{r+1}(x_i),$& 	&$b_{i,(j,l)}(y_i,y_l)=\alpha^{r+1}(x_j)$&$(i\neq j,\ i\neq l,\ j\neq l)  $;\\
$	c_{j,(i,l)}(x_i,y_l)=\alpha^{r+1}(x_j)$& 
&	$	c_{j,(i,l)}(x_i,y_j)=	\alpha^{r+1}(x_l)$,&$(i\neq j,\ i\neq l,\ j\neq l)  $;\\	
			$f_{j}(x_j,y_p)=\alpha^{r+1}(y_p),$& 
			&$f_{j}(x_p,z)=y_{p,j}\alpha^{r}(z)$&$ (\forall p\in \{1,\cdots,m\}) $;\\
			$h_{j}(x_p,x_j)=\alpha^{r+1}(x_p)$,& 
	&	$h_{j}(x_p,z)=y_{p,j}\alpha^{r}(z)$&$ (\forall p\in \{1,\cdots,m\}) ;$\\
$ k_{j}(x_p,y_j)=\alpha^{r+1}(x_p)$,& 
&$ k_{j}(y_p,z))=x_{p,j}\alpha^{r}(z)$& $( \forall p\in \{1,\cdots,m\})  $:\\
	$	a_{j}(y_j,y_p)=\alpha^{r+1}(y_p)$,& 
	&$a_{j}(y_p,z)=x_{p,j}	\alpha^{r}(z),$& $ (\forall p\in \{1,\cdots,m\}) , $
\end{tabular} 
	The image of the other basic elements is zero or deduced by skew symmetric.	
\end{thm}
\begin{cor}
	\[ \dim H^2(\mathcal{H}^m_\lambda,\mathcal{H}^m_\lambda) =4m(2m^2-6m+5)\qquad (m>2).\]
\end{cor}

\end{document}